\documentclass[reqno]{amsart}
\usepackage{amscd}
\usepackage{amssymb}
\usepackage[margin=1in,includeheadfoot]{geometry}
\usepackage{cite}
\usepackage [latin1]{inputenc}
\usepackage{tabularx,booktabs,tikz}
\usepackage{caption}
\usepackage{amsmath}
\usepackage{amsfonts}

\usepackage{amscd}
\usepackage{amsthm}
\usepackage{amssymb} \usepackage{latexsym}
\usepackage{eufrak}
\usepackage{euscript}
\usepackage{epsfig}
\usepackage{graphics}
\usepackage{array}
\usepackage{enumerate}
\usepackage{dsfont}
\usepackage{color}
\usepackage{wasysym}
\usepackage{hyperref}
\usepackage{pdfsync}

\def\beq{\begin{equation}}
\def\eeq{\end{equation}}
\def\ba{\begin{array}}
\def\ea{\end{array}}

\newtheorem{thm}{Theorem}[section]
\newtheorem{lm}[thm]{Lemma}

\newtheorem{crl}[thm]{Corollary}

\theoremstyle{definition}
\newtheorem{rem}[thm]{Remark}

\theoremstyle{remark}

\begin{document}
\pagestyle{plain}
\title{Existence and convergence of ground state solutions for Choquard-type systems on lattice graphs}

\author{Lidan Wang}
\email{wanglidan@ujs.edu.cn}
\address{Lidan Wang: School of Mathematical Sciences, Jiangsu University, Zhenjiang 212013, China}

\begin{abstract}
In this paper, we study the $p$-Laplacian system with  Choquard-type nonlinearity
$$
\begin{cases}-\Delta_{p} u+(\lambda a+1)|u|^{p-2} u=\frac{1}{\gamma} \left(R_\alpha\ast F(u,v)\right)F_{u}(u, v), \\ -\Delta_{p} v+(\lambda b+1)|v|^{p-2} v=\frac{1}{\gamma} \left(R_\alpha\ast F(u,v)\right)F_{v}(u, v),\end{cases}
$$
on lattice graphs $\mathbb{Z}^N$, where $\alpha \in(0,N),\,p\geq 2,\,\gamma> \frac{(N+\alpha)p}{2N},\,\lambda>0$ is a parameter and $R_{\alpha}$ is the Green's function of the discrete fractional Laplacian that behaves as the Riesz potential. Under some assumptions on the functions $a,\,b$ and $F$, we prove the existence and asymptotic behavior of ground state solutions  by the method of Nehari manifold.
\end{abstract}

\maketitle

{\bf Keywords:}  lattice graphs, $p$-Laplacian system, Choquard-type nonlinearity,  ground state solutions, Nehari manifold

\
\

{\bf Mathematics Subject Classification 2020:} 35J50, 35R02

\section{Introduction}

The existence, multiplicity and convergence of solutions for nonlinear Schr\"{o}dinger equations and Schr\"{o}dinger systems on Euclidean spaces have been extensively studied. See for examples \cite{AD,BW,BW1,F,L0} for a single nonlinear Schr\"{o}dinger equation. Moreover, we refer the readers to \cite{G,LP,LD,LL,SS} for the nonlinear Sch\"{o}dinger systems.

In recent years, there have been many works on the existence and convergence of ground state solutions for nonlinear Schr\"{o}dinger equations and Schr\"{o}dinger systems on graphs $G=(V,E)$, where $V$ is the vertex set and $E$ is the edge set. For example, Zhang and Zhao \cite{ZZ}
proved the existence and convergence of ground state solutions for the discrete nonlinear Schr\"{o}dinger equation
$$-\Delta u+(\lambda a+1)u=|u|^{p-1}u$$ under the assumptions on the potential function $a(x)$:
\begin{itemize}
    \item [$(a_1)$] $a(x)\geq 0$ and the potential well $\Omega=\{x\in V:a(x)=0\}$ is a non-empty, connected and bounded domain in $V$;
    \item[$(a_2)$] there exists $x_0\in V$ such that $a(x)\rightarrow 0$ as $d(x,x_0)\rightarrow\infty.$ 
\end{itemize}
After that, Han, Shao and Zhao \cite{HSZ} generalized the result in \cite{ZZ} to the discrete biharmonic Schr\"{o}dinger equation
$$\Delta^2u-\Delta u+(\lambda a+1)u=|u|^{p-1}u$$
under the conditions $(a_1)$ and $(a_2).$

Later, under the same conditions on $a(x)$, for $p\geq 2$, Han and Shao \cite{HS} studied the discrete $p$-Laplacian equation
$$-\Delta_p u+(\lambda a+1)u^{p-2}u=f(x,u),$$
where $\Delta_p u=\frac{1}{2\mu(x)}\sum\limits_{y\sim x}\omega_{xy}(|\nabla u|^{p-2}(y)+|\nabla u|^{p-2}(x))(u(y)-u(x)),$ and showed the existence and asymptotic behavior of ground state solutions.

Recently, Xu and Zhao \cite{XZ} investigated the existence and convergence of ground state solutions for the discrete Schr\"{o}dinger system 
$$\begin{cases}-\Delta u+(\lambda a+1) u=\frac{\alpha}{\alpha+\beta}|u|^{\alpha-2}u|v|^\beta, \\ -\Delta v+(\lambda b+1)v=\frac{\beta}{\alpha+\beta}|u|^{\alpha}|v|^{\beta-2}v,\end{cases}$$
where $\alpha>1,\,\beta>1$ and  the potential functions $a(x)$ and $b(x)$ satisfy
\begin{itemize}
    \item[$(A_{1})$] $a(x) \geq 0,\, b(x) \geq 0$, the potential wells $\Omega_{a}=\{x\in V: a(x)=0\},\, \Omega_{b}=\{x\in V: b(x)=0\}$ and $\Omega:=\Omega_{a} \cap \Omega_{b}$ are all non-empty, connected and bounded domains in $V$;
    \item[($A_2$)] there exists $x_0\in V$ such that $a(x)\rightarrow 0$ and $b(x)\rightarrow 0$ as $d(x,x_0)\rightarrow\infty.$ 
\end{itemize}
Under the assumptions $(A_1)$ and $(A_2)$, Shao \cite{S} considered the $p$-Laplacian system
$$\begin{cases}-\Delta_{p} u+(\lambda a+1)|u|^{p-2} u=\frac{1}{\gamma}F_{u}(u, v) , \\ -\Delta_{p} v+(\lambda b+1)|v|^{p-2} v=\frac{1}{\gamma}F_{v}(u, v),\end{cases}$$
where $F$ is a positively homogeneous function of degree $\gamma$ and $2\leq p< \gamma<\infty$, and obtained 
the existence and concentration behavior of ground state solutions. Furthermore, if $(A_2)$ is replaced by
\begin{itemize}
    \item [$(\tilde{A}_2)$] $(a(x)+1)^{-1}\in \ell^{\frac{1}{p-1}}(V)$ and $(b(x)+1)^{-1}\in \ell^{\frac{1}{q-1}}(V)$,
\end{itemize}
then under the assumptions $(A_1)$-$(\tilde{A}_2)$ and some conditions on $F$, Zhang and Zhang \cite{ZZ1} studied the $(p,q)$-Laplacian system on graphs
$$\begin{cases}-\Delta_{p} u+(\lambda a+1)|u|^{p-2} u=F_{u}(x,u, v) , \\ -\Delta_{q} v+(\lambda b+1)|v|^{q-2} v=F_{v}(x,u, v),\end{cases}$$
where $p,q >1$, and derived the existence and convergence of ground state solutions. For more related works about the Schr\"{o}diner equations on graphs, we refer the readers to \cite{GL,HLW,M,SYZ,W4}.

Nowadays, the discrete  nonlinear Choquard equation has attracted much attention from researchers. For example, under the hypotheses $(a_1)$ and $(a_2)$, Wang, Zhu and Wang \cite{WZW} established the
the existence and convergence of ground state solutions for the discrete nonlinear Choquard equation
$$-\Delta u+(\lambda a+1)u=(R_{\alpha}\ast|u|^{p})|u|^{p-2}u,$$
where $R_\alpha$ is the Green's function of the discrete fractional Laplacian. Now let
\begin{itemize}
    \item [$(\tilde{a}_{2})$] there exists $M>0$ such that $\{x\in V: a(x)\leq M\}$ is a finite and non-empty set in $V$.
\end{itemize}
Clearly, the condition $(\tilde{a}_2)$ is weaker than $(a_2)$. Li and Wang \cite{LW} established the existence and asymptotic behavior of ground state solutions for the discrete nonlinear Choquard equation
$$\Delta^{2}u-\Delta u+(\lambda a+1)u=(R_{\alpha}\ast|u|^{p})|u|^{p-2}u.$$
For more related works about the single nonlinear Choquard equation on graphs, we refer the readers to \cite{LZ1,LZ2,W1,W2,W3}. 

However, as far as we know, there are no such results for the nonlinear Choquard-type system on graphs. Inspired by the works mentioned above, in this paper, we study the existence and asymptotic behavior of ground state solutions for a class of $p$-Laplacian system with Choquard-type nonlinearity on lattice graphs $V=\mathbb{Z}^N$.

Let us first give some notations. Let $\Omega$ be a subset of $\mathbb{Z}^N$, we denote by $C(\Omega)$ the set of all functions on $\Omega$. The support of $u\in C(\Omega)$ is
defined as $\text{supp}(u):=\{x \in \Omega : u(x)\neq 0\}$. Let $C_{c}(\Omega)$ be the set of all functions with finite support on $\Omega$. Moreover, we denote by the $\ell^p(\Omega)$ the space of $\ell^p$-summable functions on $\Omega$. For convenience, for any $u\in C(\Omega)$, we always write
$
\int_{\Omega}u\mathrm{~d}\mu:=\sum\limits_{x\in \Omega}u(x),$ where $\mu$ is the counting measure on $\Omega$.

In this paper, we consider the following $p$-Laplacian system of the form
\begin{equation}\label{00}
\begin{cases}-\Delta_{p} u+(\lambda a+1)|u|^{p-2} u=\frac{1}{\gamma} \left(R_\alpha\ast F(u,v)\right)F_{u}(u, v) , \\ -\Delta_{p} v+(\lambda b+1)|v|^{p-2} v=\frac{1}{\gamma} \left(R_\alpha\ast F(u,v)\right)F_{v}(u, v),\end{cases}
\end{equation}
on lattice graph $\mathbb{Z}^N$, where $\alpha\in(0,N),\,p\geq 2$,\,$\gamma>\frac{(N+\alpha)p}{2N}$ and $ \lambda>0$ is a parameter. Here $\Delta_{p}$ is the discrete $p$-Laplacian defined as $$
\Delta_{p} u(x)= \sum_{y \sim x}|u(y)-u(x)|^{p-2}(u(y)-u(x)),
$$ and $R_\alpha$ is the Green's function of the discrete fractional Laplacian  defined by
$$R_{\alpha}(x,y)=\frac{1}{|\Gamma(-\alpha)|}\int_{0}^{\infty}k_t(x,y)\frac{dt}{t^{1-\alpha}},$$
where $\Gamma$ represents the Gamma function and $k_t(x,y)$ is the heat kernel of discrete Laplacian. We always assume that the functions $a, b$ and $F$ satisfy 
\begin{itemize}
    \item[$(A_{1})$] $a(x) \geq 0,\, b(x) \geq 0$, the potential wells $\Omega_{a}=\{x\in\mathbb{Z}^N: a(x)=0\},\, \Omega_{b}=\{x\in \mathbb{Z}^N: b(x)=0\}$ and $\Omega:=\Omega_{a} \cap \Omega_{b}$ are all non-empty bounded domains in $\mathbb{Z}^N$;
    \item[($A'_2$)] there exist $M_1,M_2>0$ such that the sets $\{x\in \mathbb{Z}^N: a(x)\leq M_1\}$ and  $\{x\in \mathbb{Z}^N: b(x)\leq M_2\}$ are finite and non-empty;
    \item[($F_1$)] $F\in C^{1}(\mathbb{R}^{2}, \mathbb{R}^{+})$ and $F(t u, t v)=t^{\gamma} F(u, v)$ for $(u, v) \in \mathbb{R}^{2}$, where $t>0$ and $\gamma>\frac{(N+\alpha)p}{2N}$.

\end{itemize}
It follows from ($F_{1}$) that, see \cite{LL,S}, 
\begin{itemize}
    \item [(i)] $u F_{u}(u, v)+v F_{v}(u, v)=\gamma F(u, v)$;
    \item[(ii)] for $(u, v) \in \mathbb{R}^{2}$, there exists $M_{F}>0$ such that
\begin{equation}\label{04}
|F(u, v)| \leq M_{F}\left(|u|^{\gamma}+|v|^{\gamma}\right), 
\end{equation}
where $M_{F}=\max \left\{F(u, v):|u|^{\gamma}+|v|^{\gamma}=1,(u, v) \in \mathbb{R}^{2}\right\}$;
\item[(iii)] $F_{u}(u, v), F_{v}(u, v)$ are positively homogeneous of degree $(\gamma-1)$.
\end{itemize}

 Let $W^{1,p}(\mathbb{Z}^N)$ be the completion of $C_c(\mathbb{Z}^N)$ with respect to the norm
$$
\|u\|_{W^{1, p}}=\left(\int_{\mathbb{Z}^N}\left(|\nabla u|^{p}+|u|^{p}\right)\mathrm{~d}\mu\right)^{\frac{1}{p}},
$$ where $$|\nabla u|^{p}:=|\nabla u|^{p}_p=\frac{1}{2}\sum\limits_{y\sim x}|u(y)-u(x)|^p,$$ see Section 2 for more details. Moreover, one gets easily that
$$\int_{\mathbb{Z}^N}|\nabla u|^p\mathrm{~d}\mu=\frac{1}{2}\sum_{x\in\mathbb{Z}^N}\sum\limits_{y\sim x}|u(y)-u(x)|^p\leq C_{N,p}\sum_{x\in\mathbb{Z}^N}|u(x)^p|=C_{N,p}\int_{\mathbb{Z}^N}|u|^p\mathrm{~d}\mu.$$
Hence we have that
$$\|u\|_p\leq \|u\|_{W^{1,p}}\leq (1+C_{N,p})\|u\|_p,$$
which implies that $\|\cdot\|_{W^{1,p}}$ and $\|\cdot\|_p$ are equivalent norms. Since $W^{1,p}(\mathbb{Z}^N)$ and $\ell^p(\mathbb{Z}^N)$ are the completion of $C_c(\mathbb{Z}^N)$ under the corresponding norms, we obtain that $W^{1,p}(\mathbb{Z}^N)=\ell^p(\mathbb{Z}^N)$, and hence $W^{1, p}(\mathbb{Z}^N)$ is a reflexive Banach space. 

For any function $h(x) \geq 0$ and $\lambda>0$, we define a subspace of $W^{1, p}(V)$, which is also a reflexive Banach space, 
$$
W_{\lambda, h}:=\left\{u \in W^{1, p}(V): \int_{\mathbb{Z}^N}(\lambda h+1)|u|^{p} \mathrm{~d} \mu<\infty\right\}
$$
under the norm
$$
\|u\|_{\lambda, h}=\left(\int_{\mathbb{Z}^N}\left(|\nabla u|^{p}+(\lambda h+1)|u|^{p}\right) \mathrm{~d} \mu\right)^{\frac{1}{p}}.
$$

Let $W_{\lambda}$ be the product space $W_{\lambda, a} \times W_{\lambda, b}$ with respecct to the norm
$$
\|(u, v)\|_{\lambda}=\left(\|u\|_{\lambda, a}^{p}+\|v\|_{\lambda, b}^{p} \right)^{p}.
$$
Clearly, $W_\lambda$ is also a reflexive Banach space.

The energy functional $J_{\lambda}(u,v): W_{\lambda}\rightarrow\mathbb{R}$ associated to the system (\ref{00}) is given by
$$
J_{\lambda}(u, v)=\frac{1}{p}\|(u, v)\|_{\lambda}^{p}-\frac{1}{2\gamma} \int_{\mathbb{Z}^N} \left(R_\alpha\ast F(u,v)\right)F(u, v) \mathrm{~d} \mu .
$$
Moreover, one gets easily that the functional $J_{\lambda} \in C^{1}(W_{\lambda}, \mathbb{R})$ and
$$
\left\langle J_{\lambda}^{\prime}(u, v),(u, v)\right\rangle=\|(u, v)\|_{\lambda}^{p}-\int_{\mathbb{Z}^N}\left(R_\alpha\ast F(u,v)\right)F(u, v) \mathrm{~d} \mu.
$$
We define the Nehari manifold as $$
\mathcal{N}_{\lambda}:=\left\{u \in W_\lambda \backslash\{0,0\}:\left\langle J_{\lambda}^{\prime}(u, v),(u, v)\right\rangle=0\right\}
.$$
We say that $(u,v)\in W_\lambda$ is a ground state solution to the system (\ref{00}), if $(u,v)$ is a nontrivial critical point of the energy functional $J_\lambda$ such that
$$
J_{\lambda}(u, v)=\inf _{\mathcal{N}_{\lambda}} J_{\lambda}=:m_{\lambda}.
$$

Now we state our first result, which is about the existence of ground state solutions to the system (\ref{00}).

\begin{thm}\label{th1}
  Let $(F_1)$ and $(A_1)$-$(A'_2)$ hold. Then there exists a $\lambda_0>0$ such that for any $\lambda\geq \lambda_0$ and $p \geq 2$, the system (\ref{00}) has a ground state solution $\left(u_{\lambda}, v_{\lambda}\right)$.  
\end{thm}

Let $\Omega\subset\mathbb{Z}^N$ be a bounded domain. We define the vertex boundary of $\Omega$ by $$\partial\Omega=\{y\in \mathbb{Z}^N, y\not\in\Omega:\exists~x\in\Omega~ \text{such~that}~y\sim x \}.$$
We denote $\bar{\Omega}:=\Omega\cup\partial\Omega$ and $\Omega^c=\mathbb{Z}^N\backslash\Omega$. In order to study the convergence of $(u_\lambda,v_\lambda)$ as $\lambda\rightarrow\infty$, we consider the following system  \begin{equation}\label{01}
\begin{cases}-\Delta_{p} u+|u|^{p-2} u=\frac{1}{\gamma}\left(R_\alpha\ast F(u,v)\right) F_{u}(u, v), & x\in\Omega_{a},  \\ -\Delta_{p} v+|v|^{p-2} v=\frac{1}{\gamma} \left(R_\alpha\ast F(u,v)\right)F_{v}(u, v), & x\in\Omega_{b}, \\ u =0, &x\in\partial\Omega_{a},\\  v =0, &x\in\partial\Omega_{b}.
\end{cases}
\end{equation}

Let $W_{0}^{1, p}(\Omega)$ be the completion of $C_{c}(\Omega)$ under the norm
$$
\|u\|_{W_{0}^{1, p}(\Omega)}=\left(\int_{\bar{\Omega}}|\nabla u|^{p} \mathrm{~d} \mu+\int_{\Omega}|u|^{p} \mathrm{~d} \mu\right)^{\frac{1}{p}} .
$$
Moreover, we define $W_{\Omega}$ as the product space $W_{0}^{1, p}\left(\Omega_{a}\right) \times W_{0}^{1, p}\left(\Omega_{b}\right)$
under the inner product
$$
\begin{aligned}
&\langle(u, v),(\phi, \psi)\rangle_{W_{\Omega}}\\= & \int_{\bar{\Omega}_{a} \cup \bar{\Omega}_{b}}\left(|\nabla u|^{p-2} \nabla u \nabla \phi+|\nabla v|^{p-2} \nabla v \nabla \psi\right) \mathrm{~d} \mu \\
& +\int_{\Omega_{a} \cup \Omega_{b}}\left(|u|^{p-2} u \phi+|v|^{p-2} v \psi\right) \mathrm{~d} \mu,\quad (u,v),(\phi,\psi)\in W_\Omega.
\end{aligned}
$$

The energy functional $J_{\Omega}(u,v): W_\Omega\rightarrow\mathbb{R}$ related to the system (\ref{01}) is
\begin{align*}
J_{\Omega}(u, v)= & \frac{1}{p} \int_{\bar{\Omega}_{a} \cup \bar{\Omega}_{b}}\left(|\nabla u|^{p}+|\nabla v|^{p}\right) \mathrm{~d} \mu \\
& +\frac{1}{p} \int_{\Omega_{a} \cup \Omega_{b}}\left(|u|^{p}+|v|^{p}\right) \mathrm{~d} \mu-\frac{1}{2\gamma} \int_{\Omega_{a} \cup \Omega_{b}} \left(R_\alpha\ast F(u,v)\right)F(u, v)  \mathrm{~d} \mu .
\end{align*}
Similarly, $(u,v)\in W_\Omega$ is a ground state solution of the system (\ref{01}) if
$(u,v)$ is a nontrivial critical point of $J_\Omega$ such that $$J_\Omega(u,v)=\inf\limits_{\mathcal{N}_\Omega} J_\Omega:=m_\Omega,$$
where $\mathcal{N}_{\Omega}=\{u\in W_\Omega\backslash\{(0,0)\}: (J'_\Omega(u,v),(u,v))=0\}$.

Our second result is about the asymptotic behavior of ground state solutions as $\lambda\rightarrow\infty$.

\begin{thm}\label{th4}
 Let $(F_1)$ and $(A_1)$-$(A'_2)$ hold.  Then for any sequence $\lambda_{k} \rightarrow \infty$, up to a subsequence, the corresponding ground state solutions $(u_{k}, v_{k})$ of the system (\ref{00}) converge in $W^{1, p}(V) \times W^{1, p}(V)$ to a ground state solution of the system (\ref{01}).   
\end{thm}

\begin{rem}
  \begin{itemize}
      \item [(i)] In this paper, the definition of $p$-Laplacian is different from that in \cite{HS,S}, and thus the formula of integration by parts does not work for our results. Luckily, we also give a formula of integration by parts based on our definition of $p$-Laplacian, see Section 2;
      \item[(ii)] The authors in \cite{XZ,S} established the existence and convergence of ground state solutions under the assumptions on $(A_1)$ and $(A_2)$, and Zhang-Zhang \cite{ZZ1} obtained similar results under the conditions on $(A_1)$ and $(\tilde{A}_2)$. Note that the condition $(A_2)$ or $(\tilde{A}_2)$ guarantees a compact embedding, which plays a key role in their papers. However the condition $(A'_2)$ is weaker than the condition $(A_2)$ or $(\tilde{A}_2)$, which leads to the lack of compactness. Therefore, we have to seek for other method to overcome this difficulty;
      \item[(iii)] To the best of our knowledge, this is a first work to study the Choquard-type system on graphs. Moreover, we would like to say that we can obtain similar results under the assumptions on $(A_1)$ and $(A_2)$ or $(\tilde{A}_2)$.
  \end{itemize}  
\end{rem}

This paper is organized as follows. In Section 2, we state  some basic results in this paper. In Section 3, we establish the existence of ground state solutions to the system (\ref{00})(Theorem \ref{th1}). In Section 4, we prove the convergence of the ground state solutions of the system (\ref{00})(Theorem \ref{th4}).

\section{ Preliminaries}
In this section, we state some basic results on graphs. Let $G=(V, E)$ be a connected, locally finite graph, where $V$ denotes the vertex set and $E$ denotes the edge set. We call vertices $x$ and $y$ neighbors, denoted by $x \sim y$, if there exists an edge connecting them, i.e. $(x, y) \in E$. For any $x,y\in V$, the distance $d(x,y)$ is defined as the minimum number of edges connecting $x$ and $y$, namely
$$d(x,y)=\inf\{k:x=x_0\sim\cdots\sim x_k=y\}.$$
Let $B_{r}(a)=\{x\in V: d(x,a)\leq r\}$ be the closed ball of radius $r$ centered at $a\in V$. For brevity, we write $B_{r}:=B_{r}(0)$.

In this paper, we consider, the natural discrete model of the Euclidean space, the integer lattice graph.  The $N$-dimensional integer lattice graph, denoted by $\mathbb{Z}^N$, consists of the set of vertices $V=\mathbb{Z}^N$ and the set of edges $E=\{(x,y): x,\,y\in\mathbb{Z}^N,\,\underset {{i=1}}{\overset{N}{\sum}}|x_{i}-y_{i}|=1\}.$
In the sequel, we denote $|x-y|:=d(x,y)$ on the lattice graph $\mathbb{Z}^N$.

Let $C(\mathbb{Z}^N)$ be the set of all functions on $\mathbb{Z}^N$. The Laplacian of $u\in C(\mathbb{Z}^N)$ is defined as
\begin{eqnarray*}
\Delta u(x)=\underset {y\sim x}{\sum}\left(u(y)-u(x)\right).
\end{eqnarray*} 
The associated gradient form is given by
$$
\Gamma(u, v)(x)=\frac{1}{2} \sum\limits_{y \sim x}(u(y)-u(x))(v(y)-v(x)):=\nabla u\nabla v.
$$
We write $\Gamma(u)=\Gamma(u, u)$ and denote the length of this gradient as
$$
|\nabla u(x)|=\sqrt{\Gamma(u)(x)}=\left(\frac{1}{2} \sum\limits_{y \sim x}(u(y)-u(x))^{2}\right)^{\frac{1}{2}}.
$$

For $p\geq 2$, we define the $p$-Laplacian of $u\in C(\mathbb{Z}^N)$ by
\begin{equation}\label{k5}
\Delta_{p} u(x)= \sum_{y \sim x}|u(y)-u(x)|^{p-2}(u(y)-u(x)),
\end{equation} 
and the associated gradient form by
$$|\nabla u|^{p-2}\nabla u\nabla v=:\frac{1}{2}\sum\limits_{y\sim x}|u(y)-u(x)|^{p-2}(u(y)-u(x))(v(y)-v(x)).$$
Moreover, we denote the $p$-norm of the gradient as $$|\nabla u(x)|_p:=\left(\frac{1}{2}\sum\limits_{y\sim x}|u(y)-u(x)|^p\right)^{\frac{1}{p}}.$$
Clearly, for $p=2$, we get the usual Laplacian and the length of the gradient on lattice graphs $\mathbb{Z}^N$.

The space $\ell^{p}(\mathbb{Z}^N)$ is defined as $
\ell^{p}(\mathbb{Z}^N)=\left\{u \in C(\mathbb{Z}^N):\|u\|_{p}<\infty\right\},
$ where
$$
\|u\|_{p}= \begin{cases}\left(\sum\limits_{x \in \mathbb{Z}^N}|u(x)|^{p}\right)^{\frac{1}{p}}, &  1 \leq p<\infty, \\ \sup\limits_{x \in \mathbb{Z}^N}|u(x)|, & p=\infty.\end{cases}
$$

Note that the definition of the operator $\Delta_{p}$  (\ref{k5}) differs from that in \cite{HS}. In the following, we establish the formula of integration by parts for this $p$-Laplacian on graphs. 

\begin{lm}\label{l1}
 Let $u \in W^{1, p}(V)$. Then for any $v \in C_{c}(V)$, we have

$$
\int_{V}|\nabla u|^{p-2} \nabla u \nabla v \mathrm{~d} \mu=-\int_{V}(\Delta_{p} u) v \mathrm{~d} \mu .
$$   
\end{lm}

\begin{proof}
For any $v\in C_c(\mathbb{Z}^N)$, we have that
$$\begin{aligned}
\int_{V}(\Delta_{p} u) v \mathrm{~d} \mu=&\sum\limits_{x\in V}\left[\sum\limits_{y\sim x}|u(y)-u(x)|^{p-2}(u(y)-u(x))\right]v(x)\\=& \sum\limits_{y\in V}\left[\sum\limits_{x\sim y}|u(x)-u(y)|^{p-2}(u(x)-u(y))\right]v(y).   
\end{aligned}$$
Adding together the last two lines and dividing by 2, we obtain that
$$\begin{aligned}
\int_{V}(\Delta_{p} u) v \mathrm{~d} \mu=&\frac{1}{2}\sum\limits_{x\in V}\sum\limits_{y\sim x}|u(y)-u(x)|^{p-2}(u(y)-u(x))(v(x)-v(y))\\=&-\frac{1}{2}\sum\limits_{x\in V}\sum\limits_{y\sim x}|u(y)-u(x)|^{p-2}(u(y)-u(x))(v(y)-v(x))\\=&-\sum\limits_{x\in V}\left[\frac{1}{2}\sum\limits_{y\sim x}|u(y)-u(x)|^{p-2}(u(y)-u(x))(v(y)-v(x))\right]\\=&- \int_{V} |\nabla u|^{p-2} \nabla u \nabla v \mathrm{~d} \mu.
\end{aligned}$$

\end{proof}

By Lemma \ref{l1}, one gets easily the following  result.
\begin{lm}\label{l2}
  Let $u \in W^{1, p}(\Omega)$. Then for any $v \in C_{c}(\Omega)$, we have $$
\int_{\bar{\Omega}}|\nabla u|^{p-2} \nabla u \nabla v \mathrm{~d} \mu=-\int_{\Omega}\left(\Delta_{p} u\right) v \mathrm{~d} \mu,
$$  
where $\Omega \subset V$ is a bounded domain. 
\end{lm}

The following discrete Hardy-Littlewood-Sobolev (HLS for abbreviation) inequality is  well-known, we refer the readers to \cite{LW,W1}.

\begin{lm}\label{l8}
Let $0<\alpha <N,\,1<r,s<\infty$ and $\frac{1}{r}+\frac{1}{s}+\frac{N-\alpha}{N}=2$. We have the discrete
HLS inequality
\begin{equation}\label{bo}
\int_{V}(R_\alpha\ast u)(x)v(x)\mathrm{~d}\mu\leq C_{r,s,\alpha,N}\|u\|_r\|v\|_s,\quad u\in \ell^r(V),\,v\in \ell^s(V).
\end{equation}
And an equivalent form is
\begin{equation}\label{p}
\|R_\alpha\ast u\|_{\frac{Nr}{N-\alpha r}}\leq C_{r,\alpha,N}\|u\|_r,\quad u\in \ell^r(V),
\end{equation}
where $0<\alpha <N,\,1<r<\frac{N}{\alpha}$.

\end{lm}

\begin{crl}
    Let $(F_1)$ hold. Then we have that
    \begin{equation}\label{p1}
    \int_{V} \left(R_\alpha\ast F(u,v)\right)F(u, v) \mathrm{~d} \mu \leq C_{N,\alpha,M_F}\|(u,v)\|^{2\gamma}_{\lambda}.    
    \end{equation}
\end{crl}
\begin{proof}
Let $r=s=\frac{2N}{N+\alpha}$ in (\ref{bo}). Note that $\frac{2N\gamma}{N+\alpha}>p$, by (\ref{04}), we get that
$$
\begin{aligned}
 \int_{V} \left(R_\alpha\ast F(u,v)\right)F(u, v) \mathrm{~d} \mu \leq & C_{N,\alpha}\left(\int_V|F(u,v)|^{\frac{2N}{N+\alpha}}\mathrm{~d}\mu\right)^{\frac{N+\alpha}{N}} \\ \leq& C_{N,\alpha}\left(\int_V\left[M_{F}(|u|^{\gamma}+|v|^{\gamma})\right]^{\frac{2N}{N+\alpha}}\mathrm{~d}\mu\right)^{\frac{N+\alpha}{N}}\\ \leq& C_{N,\alpha,M_F}\left(\int_V\left(|u|^{\frac{2N\gamma}{N+\alpha}}+|v|^{\frac{2N\gamma}{N+\alpha}}\right)\mathrm{~d}\mu\right)^{\frac{N+\alpha}{N}}\\\leq &C_{N,\alpha,M_F}\left(\|u\|^{2\gamma}_{\frac{2N\gamma}{N+\alpha}}+\|v\|^{2\gamma}_{\frac{2N\gamma}{N+\alpha}}\right)\\\leq &C_{N,\alpha,M_F}\left(\|u\|^{2\gamma}_{p}+\|v\|^{2\gamma}_{p}\right)\\ \leq &C_{N,\alpha,M_F}\|(u,v)\|^{2\gamma}_{\lambda}.
\end{aligned}
$$
\end{proof}

\begin{lm}\label{pa}
 Let $(F_1)$ hold. If $\left\{\left(u_{k}, v_{k}\right)\right\}$ is a bounded sequence in $\ell^{\frac{2N\gamma}{N+\alpha}}(V,\mathbb{R}^2)$ and $\left(u_{k}, v_{k}\right) \rightarrow (u, v)$ pointwise in $V$, then we have that
\begin{equation*}
\int_{V} \left(R_\alpha\ast F(u_k,v_k)\right)F(u_k, v_k) \mathrm{~d} \mu-\int_{V} \left(R_\alpha\ast F(w_k,z_k)\right)F(w_k, z_k) \mathrm{~d} \mu=\int_{V} \left(R_\alpha\ast F(u,v)\right)F(u, v) \mathrm{~d} \mu+o_{k}(1), 
\end{equation*} 
where $w_k=u_k-u$ and $z_k=v_k-v.$
\end{lm}

\begin{proof}
First, we prove that
\begin{equation}\label{p5}
 \lim\limits_{k \rightarrow \infty} \int_{V}\left|F\left(u_{k}, v_{k}\right)-F\left(w_{k}, z_{k}\right)-F(u, v)\right| ^{\frac{2N}{N+\alpha}}\mathrm{~d} \mu=0. 
\end{equation}
In fact, by the mean value theorem, we obtain that
$$
\begin{aligned}
\left|F\left(u_{k}, v_{k}\right)-F\left(w_{k}, z_{k}\right)\right|= & \mid \nabla F(w_{k}+\theta u, z_{k}+\theta v) \cdot(u, v) \mid, \quad \theta \in(0,1) .
\end{aligned}
$$
Since $F_{u}(u, v), F_{v}(u, v)$ are positively homogeneous of degree $(\gamma-1)$, we get that
$$
\begin{aligned}
\left|F\left(u_{k}, v_{k}\right)-F\left(w_{k}, z_{k}\right)\right| 
=&\left|\nabla F\left(w_{k}+\theta u, z_{k}+\theta v\right) \cdot(u, v)\right| \\
\leq &C\left(\left|w_{k}+\theta u\right|^{\gamma-1}+\left|z_{k}+\theta v\right|^{\gamma-1}\right)|u| \\
&+C\left(\left|w_{k}+\theta u\right|^{\gamma-1}+\left|z_{k}+\theta v\right|^{\gamma-1}\right)|v| \\
\leq &C'\left(\left|w_{k}\right|^{\gamma-1}+|u|^{\gamma-1}+\left|z_{k}\right|^{\gamma-1}+|v|^{\gamma-1}\right)|u| \\
&+C'\left(\left|w_{k}\right|^{\gamma-1}+|u|^{\gamma-1}+\left|z_{k}\right|^{\gamma-1}+|v|^{\gamma-1}\right)|v|.
\end{aligned}
$$
For the last inequality, given any $\varepsilon>0$, by the Young inequality, there exists $C_{\varepsilon}>0$ such that
\begin{equation*}\label{p3}
\left|F\left(u_{k}, v_{k}\right)-F\left(w_k, z_{k}\right)\right| \leq \varepsilon\left(\left|w_k\right|^{\gamma}+\left|z_k\right|^{\gamma}\right)+C_{\varepsilon}\left(|u|^{\gamma}+|v|^{\gamma}\right).
\end{equation*}
Therefore, we obtain that
$$
\begin{aligned}
&\left|F\left(u_{k}, v_{k}\right)-F\left(w_k, z_k\right)-F(u, v)\right|^{\frac{2N}{N+\alpha}}\\ \leq& C_{N,\alpha}\left(|F\left(u_{k}, v_{k}\right)-F(w_k, z_k)|^{\frac{2N}{N+\alpha}} +|F(u,v)|^{\frac{2N}{N+\alpha}}\right)\\ \leq &C_{N,\alpha}\left[\varepsilon\left(\left|w_{k}\right|^{\gamma}+\left|z_{k}\right|^{\gamma}\right)+C_{\varepsilon}\left(|u|^{\gamma}+|v|^{\gamma}\right)\right]^{\frac{2N}{N+\alpha}}+C_{N,\alpha}\left[M_{F}(|u|^{\gamma}+|v|^{\gamma})\right]^{\frac{2N}{N+\alpha}}\\ \leq &\varepsilon\left(\left|w_{k}\right|^{\gamma}+\left|z_{k}\right|^{\gamma}\right)^{\frac{2N}{N+\alpha}}+C_{\varepsilon,N,\alpha}\left(|u|^{\gamma}+|v|^{\gamma}\right)^{\frac{2N}{N+\alpha}}+C_{M_F,N,\alpha}(|u|^{\gamma}+|v|^{\gamma})^{\frac{2N}{N+\alpha}}\\ \leq & \varepsilon\left(\left|w_{k}\right|^{\frac{2N\gamma}{N+\alpha}}+\left|z_{k}\right|^{\frac{2N\gamma}{N+\alpha}}\right)+C_{\varepsilon,M_F,N,\alpha}\left(|u|^{\frac{2N\gamma}{N+\alpha}}+|v|^{\frac{2N\gamma}{N+\alpha}}\right).
\end{aligned}
$$
Denote
$$
f_{k}:=\left|F\left(u_{k}, v_{k}\right)-F\left(w_{k}, z_{k}\right)-F(u, v)\right|^{\frac{2N}{N+\alpha}}-\varepsilon\left(\left|w_{k}\right|^{\frac{2N\gamma}{N+\alpha}}+\left|z_{k}\right|^{\frac{2N\gamma}{N+\alpha}}\right).
$$
Then we have that
$$
f_{k} \leq C_{\varepsilon,M_F,N,\alpha}\left(|u|^{\frac{2N\gamma}{N+\alpha}}+|v|^{\frac{2N\gamma}{N+\alpha}}\right)\in \ell^{1}(V).
$$
Since $
\left(u_{k}, v_{k}\right) \rightarrow(u, v)$ pointwise in $V$, 
we get that $f_{k} \rightarrow 0$ pointwise in $V$. By the Lebesgue dominated convergence theorem, we obtain that
$$
\lim _{k \rightarrow \infty} \int_{V} f_{k}(x) \mathrm{~d} \mu=0.
$$
Hence, we have
$$
\begin{aligned}
\limsup _{k \rightarrow \infty} \int_{V}\left|F\left(u_{k}, v_{k}\right)-F\left(w_{k}, z_{k}\right)-F(u, v)\right| ^{\frac{2N}{N+\alpha}}\mathrm{~d} \mu 
=&\limsup _{k \rightarrow \infty} \int_{V} f_{k}(x)+\varepsilon\left(\left|w_{k}\right|^{\frac{2N\gamma}{N+\alpha}}+\left|z_{k}\right|^{\frac{2N\gamma}{N+\alpha}}\right) \mathrm{~d} \mu \\
\leq &\limsup _{k \rightarrow \infty} \int_{V} f_{k}(x) \mathrm{~d} \mu\\&+\varepsilon\limsup _{k \rightarrow \infty} \int_{V} \left(\left|w_{k}\right|^{\frac{2N\gamma}{N+\alpha}}+\left|z_{k}\right|^{\frac{2N\gamma}{N+\alpha}}\right) \mathrm{~d} \mu\\ \leq &C\varepsilon.
\end{aligned}
$$
By the arbitrariness of $\varepsilon$, we prove that (\ref{p5}) holds.

A direct calculation yields that
$$
\begin{aligned}
&\int_{V} \left(R_\alpha\ast F(u_k,v_k)\right)F(u_k, v_k) \mathrm{~d} \mu-\int_{V} \left(R_\alpha\ast F(w_k,z_k)\right)F(w_k, z_k) \mathrm{~d} \mu\\=&\int_{V} \left[R_\alpha\ast (F(u_k,v_k)-F(w_k,z_k)\right](F(u_k, v_k)-F(w_k,z_k)) \mathrm{~d} \mu\\&+2\int_{V} \left[R_\alpha\ast (F(u_k,v_k)-F(w_k,z_k)\right]F(w_k, z_k) \mathrm{~d} \mu\\=:&I_1+2I_2.    
\end{aligned}
$$
For $I_1$, by the HLS inequality (\ref{bo}) and (\ref{p5}), we get that
$$
\begin{aligned}
&\left|I_1-\int_{V} \left(R_\alpha\ast F(u,v)\right)F(u, v) \mathrm{~d} \mu\right|\\ \leq &\int_{V} \left[R_\alpha\ast \left|F(u_k,v_k)-F(w_k,z_k)-F(u, v)\right|\right]|F(u_k,v_k)-F(w_k,z_k)-F(u, v)|\mathrm{~d} \mu\\&+2 \int_{V} \left[R_\alpha\ast \left|F(u_k,v_k)-F(w_k,z_k)-F(u, v)\right|\right]|F(u, v)|\mathrm{~d} \mu\\ \leq &\left(\int_{V} \left|F(u_k,v_k)-F(w_k,z_k)-F(u, v)\right|^{\frac{2N}{N+\alpha}}\mathrm{~d} \mu\right)^{\frac{N+\alpha}{N}}\\&+\left(\int_{V} \left|F(u_k,v_k)-F(w_k,z_k)-F(u, v)\right|^{\frac{2N}{N+\alpha}}\mathrm{~d} \mu\right)^{\frac{N+\alpha}{2N}}\left(\int_{V} \left|F(u, v)\right|^{\frac{2N}{N+\alpha}}\mathrm{~d} \mu\right)^{\frac{N+\alpha}{2N}}\\ \rightarrow &0.
\end{aligned}
$$
Let $r=\frac{2N}{N+\alpha}$ in (\ref{p}), we get that
$$\|R_\alpha\ast F(u, v)\|_{\frac{2N}{N-\alpha}}\leq C_{N,\alpha}\|F(u,v)\|_{\frac{2N}{N+\alpha}}.$$
Moreover, by the boundedness of $\left\{\left(u_{k}, v_{k}\right)\right\}$ in $\ell^{\frac{2N\gamma}{N+\alpha}}(V,\mathbb{R}^2)$ and $(w_k,z_k)\rightarrow (0,0)$ pointwise in $V$, we obtain that
$$\int_V|F(w_k,v_k)|^{\frac{2N}{N+\alpha}}\mathrm{~d}\mu\leq\int_V\left[M_{F}(|w_k|^{\gamma}+|v_k|^{\gamma})\right]^{\frac{2N}{N+\alpha}}\mathrm{~d}\mu\leq C_{N,\alpha,M_F}\int_V\left(|w_k|^{\frac{2N\gamma}{N+\alpha}}+|z_k|^{\frac{2N\gamma}{N+\alpha}}\right)\mathrm{~d}\mu<\infty,$$
and
$$F(w_k,z_k)\rightharpoonup 0,\quad \text{in~}\ell^{\frac{2N}{N+\alpha}}(V,\mathbb{R}^2).$$
Then for $I_2$, we have that
$$
\begin{aligned}
|I_2|\leq&\int_{V} \left[R_\alpha\ast \left|F(u_k,v_k)-F(w_k,z_k)-F(u, v)\right|\right]|F(w_k,z_k)|\mathrm{~d} \mu +\int_{V} \left(R_\alpha\ast |F(u, v)|\right)|F(w_k,z_k)|\mathrm{~d} \mu \\ \leq & \left(\int_{V} \left|F(u_k,v_k)-F(w_k,z_k)-F(u, v)\right|^{\frac{2N}{N+\alpha}}\mathrm{~d} \mu\right)^{\frac{N+\alpha}{2N}}\left(\int_{V} \left|F(w_k, z_k)\right|^{\frac{2N}{N+\alpha}}\mathrm{~d} \mu\right)^{\frac{N+\alpha}{2N}}\\ &+ \int_{V} \left(R_\alpha\ast F(u, v)\right)F(w_k,z_k)\mathrm{~d} \mu\\ \rightarrow& 0.
\end{aligned}
$$
Hence as $k\rightarrow\infty,$ we get that
$$\int_{V} \left(R_\alpha\ast F(u_k,v_k)\right)F(u_k, v_k) \mathrm{~d} \mu-\int_{V} \left(R_\alpha\ast F(w_k,z_k)\right)F(w_k, z_k) \mathrm{~d} \mu\rightarrow \int_{V} \left(R_\alpha\ast F(u,v)\right)F(u, v) \mathrm{~d} \mu.$$
\end{proof}

\begin{lm}\label{lm16} \it (Lions lemma)
Let $2\leq p<\infty$. Assume that $\{u_k\}$ is bounded in $\ell^{p}(V)$ and $\|u_{k}\|_{\infty}\rightarrow0$ as $k\rightarrow\infty.$
Then for any $p<q<\infty$, as $k\rightarrow\infty,$
\begin{eqnarray*}
u_k\rightarrow0,\qquad \text{in}~\ell^{q}(V).
\end{eqnarray*}
\end{lm}
\begin{proof}
For $p<q<\infty$, this result follows from the interpolation inequality
\begin{eqnarray*}
\|u_k\|^{q}_{q}\leq\|u_k\|_{p}^{p}\|u_k\|_{\infty}^{q-p}.
\end{eqnarray*}

\end{proof}

In the following, we state some results of the Nehari manifold $\mathcal{N}_\lambda.$

\begin{lm}\label{k9}
 Let $(F_{1})$ and $(A_{1})$-$(A'_{2})$ hold. Then for any $(u,v)\in\mathcal{N}_{\lambda}$, we have that
 \begin{itemize}
     \item [(i)] there exists $\sigma>0$ such that $\|(u,v)\|_{\lambda}\geq \sigma$;
     \item[(ii)] $m_{\lambda}=\inf\limits_{(u, v) \in \mathcal{N}_{\lambda}} J_{\lambda}(u, v)>0.$
 \end{itemize}

\end{lm}

\begin{proof}
(i) for any $(u, v) \in \mathcal{N}_{\lambda}$, by (\ref{p1}), we have that
$$
\|(u, v)\|_{\lambda}^{p}=\int_{V} \left(R_\alpha\ast F(u,v)\right)F(u, v) \mathrm{~d} \mu\leq C_{N,\alpha, M_F}\|(u,v)\|^{2\gamma}_{\lambda}.
$$
Since $2\gamma>p$, we obtain that
$$
\|(u, v)\|_{\lambda} \geq\left(\frac{1}{C_{N,\alpha, M_F}}\right)^{\frac{1}{2\gamma-p}}>0.
$$
Hence we get the proof by taking $\sigma=\left(\frac{1}{C_{N,\alpha, M_F}}\right)^{\frac{1}{2\gamma-p}}.$

(ii) It follows from (i) that
$$
\begin{aligned}
m_{\lambda} & =\inf _{(u, v) \in \mathcal{N}_{\lambda}} J_{\lambda}(u, v)\\=&\left(\frac{1}{p}-\frac{1}{2\gamma}\right) \inf _{(u, v) \in \mathcal{N}_{\lambda}}\|(u, v)\|_{\lambda}^{p} \\
\geq&\left(\frac{1}{p}-\frac{1}{2\gamma}\right)\left(\frac{1}{C_{N,\alpha,M_F}}\right)^{\frac{p}{2\gamma-p}}\\ >&0.
\end{aligned}
$$
\end{proof}

\begin{lm}\label{l9}
 Assume that $(F_1)$ and  $(A_1)$-$(A'_2)$ hold. Let $(u, v) \in W_\lambda \backslash\{(0,0)\}$ such that $\left\langle J_{\lambda}^{\prime}((u, v)), (u, v)\right\rangle\leq 0$, then there exists a unique $t_{0}\in(0,1]$ such that $t_{0}(u, v) \in \mathcal{N}_{\lambda}$. 
\end{lm}
 
\begin{proof}
Let $(u, v) \in W_\lambda \backslash\{(0,0)\}$ be fixed. For $t>0$, we define
$$
\begin{aligned}
g(t)& =\left\langle J_{\lambda}^{\prime}(t(u, v)), t(u, v)\right\rangle \\
& =t^{p}\|(u, v)\|_{\lambda}^{p}-\int_{V} \left(R_\alpha\ast F(tu,tv)\right)F(t u, t v) \mathrm{~d} \mu \\
& =t^{p}\|(u, v)\|_{\lambda}^{p}-t^{2\gamma} \int_{V} \left(R_\alpha\ast F(u,v)\right)F(u, v) \mathrm{~d} \mu.
\end{aligned}
$$
Since $2\gamma>p$ and $\left\langle J_{\lambda}^{\prime}((u, v)), (u, v)\right\rangle\leq 0$, one gets easily that $g(t)=0$ has a unique solution $t_0=\left(\frac{\|(u, v)\|_{\lambda}^{p}}{\int_{V} \left(R_\alpha\ast F(u,v)\right)F(u,v) \mathrm{~d} \mu}\right)^{\frac{1}{2\gamma-p}}\in(0,1]$. This implies that $t_{0}(u, v) \in$ $\mathcal{N}_{\lambda}$.
\end{proof}

Finally, we show that the functional $J_{\lambda}(u,v)$ satisfies the mountain-pass geometry.
\begin{lm}\label{lm6}
Let $(F_1)$ and $(A_1)$-$(A'_2)$ hold. Then
\begin{itemize}
\item[(i)] there exist $\theta, \rho>0$ such that $J_{\lambda}(u,v)\geq\theta>0$ for $\|(u,v)\|_{\lambda}=\rho$;
\item[(ii)] there exists $(u_0,v_0)\in W_\lambda$ with $\|(u_0,v_0)\|_{\lambda}>\rho$ such that $J_{\lambda}(u_0,v_0)< 0$.
\end{itemize}

\end{lm}
\begin{proof}
(i) By  (\ref{p1}), we get that
$$\begin{aligned}
    J_{\lambda}(u, v)=&\frac{1}{p}\|(u, v)\|_{\lambda}^{p}-\frac{1}{2\gamma} \int_{V} \left(R_\alpha\ast F(u,v)\right)F(u, v) \mathrm{~d} \mu\\ \geq&\frac{1}{p}\|(u, v)\|_{\lambda}^{p}-\frac{1}{2\gamma}C_{N,\alpha,M_F}\|(u,v)\|^{2\gamma}_{\lambda}.
\end{aligned}$$
Since $2\gamma>p$, there exist $\theta>0$ and $\rho>0$ small enough such that $J_{\lambda}(u,v)\geq\theta>0$ for $\|(u,v)\|_{\lambda}=\rho$.

\
\

(ii)
First for each $\lambda>0,\,J_{\lambda}(0,0)=0$. Moreover, for any $(u,v)\in W_\lambda\backslash\{(0,0)\}$, as $t\rightarrow\infty$, one gets that
\begin{eqnarray*}
J_{\lambda}(t(u,v))=\frac{t^p}{p}\|(u,v)\|^{p}_{\lambda}-\frac{t^{2\gamma}}{2\gamma}\int_{V} \left(R_\alpha\ast F(u,v)\right)F(u, v) \mathrm{~d} \mu\rightarrow
-\infty.
\end{eqnarray*}
Therefore, there exists $t_0>0$ large enough such that $\|t_0 (u,v)\|_{\lambda}>\rho$ and $J_{\lambda}(t_0(u,v))<0$. By taking $(u_0,v_0)=t_0(u,v)$, we get the desired result.

\end{proof}

\section{Existence of the ground state solutions}
In this section, we prove the existence of ground state solutions of the system (\ref{00}) by the method of Nehari manifold. Recall that, for a given functional $I\in C^{1}(E,\mathbb{R})$, a sequence $\{z_k\}\subset E$ is a $(PS)_c$ sequence of the functional $I$, if it satisfies, as $k\rightarrow\infty$,
\begin{eqnarray*}
I(z_k)\rightarrow c, \qquad \text{in}~ E,\qquad\text{and}\qquad
I'(z_k)\rightarrow 0, \qquad \text{in}~ E^{*}.
\end{eqnarray*}
where $E$ is a Banach space and $E^{*}$ is the dual space of $E$. Moreover, if any $(PS)_c$ sequence has a convergent subsequence, then we say that $I$ satisfies $(PS)_c$ condition.

First, we prove some crucial results about the $(PS)_c$ sequence of the functional $J_\lambda$. 

\begin{lm}\label{pb}
Assume that $(F_1)$ and  $(A_1)$-$(A'_2)$ hold. Let $\{(u_k,v_k)\}\subset W_\lambda$ be a $(PS)_c$ sequence of the functional $J_\lambda$. Then
\begin{itemize}
 \item[(i)] $\{(u_k,v_k)\}$ is bounded in $W_\lambda$;
\item[(ii)] $\lim\limits_{k \rightarrow \infty}\left\|\left(u_{k}, v_{k}\right)\right\|_{\lambda}^{p}=\lim\limits_{k \rightarrow \infty} \int_{V} (R_\alpha\ast F(u_k,v_k))F(u_{k}, v_{k}) \mathrm{~d} \mu=\frac{2\gamma p c}{2\gamma-p},$
where either $c=0$ or $c\geq c_0$ for some $c_0>0$ not depending on $\lambda$.
\end{itemize}   
\end{lm}

\begin{proof}
(i) Let $\{(u_k,v_k)\}\subset W_\lambda$ be a $(PS)_c$ sequence of the functional $J_\lambda$, namely $J_{\lambda}\left(u_{k}, v_{k}\right) =c+o_k(1)$ and $\|J_{\lambda}^{\prime}\left(u_{k}, v_{k}\right)\|_{\lambda}=o_k(1)$. Then we have
$$\begin{aligned}
 \left(\frac{1}{p}-\frac{1}{2\gamma}\right)\left\|\left(u_{k}, v_{k}\right)\right\|_{\lambda}^{p}=&J_{\lambda}\left(u_{k}, v_{k}\right)-\frac{1}{2\gamma}\left\langle J_{\lambda}^{\prime}\left(u_{k}, v_{k}\right),\left(u_{k}, v_{k}\right)\right\rangle \\ \leq & c+o_k(1)+\frac{1}{2\gamma}o_k(1)\|\left(u_{k}, v_{k}\right)\|_{\lambda},  \end{aligned}$$
which implies that $\{(u_k,v_k)\}$ is bounded in $W_\lambda$.

(ii) Since $\{(u_k,v_k)\}$ is bounded in $W_\lambda$, we have that
$\lim\limits_{k\rightarrow\infty}\left\langle J_{\lambda}^{\prime}\left(u_{k}, v_{k}\right),\left(u_{k}, v_{k}\right)\right\rangle=0.$ Then
$$\lim\limits_{k \rightarrow \infty}\left(\frac{1}{p}-\frac{1}{2\gamma}\right)\left\|\left(u_{k}, v_{k}\right)\right\|_{\lambda}^{p}=\lim\limits_{k \rightarrow \infty}\left[J_{\lambda}\left(u_{k}, v_{k}\right)-\frac{1}{2\gamma}\left\langle J_{\lambda}^{\prime}\left(u_{k}, v_{k}\right),\left(u_{k}, v_{k}\right)\right\rangle\right]=c.$$
Moreover, we have that
$$
\lim\limits_{k \rightarrow \infty}\left(\frac{1}{p}-\frac{1}{2\gamma}\right) \int_{V} (R_\alpha\ast F(u_k,v_k))F(u_{k}, v_{k}) \mathrm{~d} \mu=\lim _{k \rightarrow \infty}\left[J_{\lambda}\left(u_{k}, v_{k}\right)-\frac{1}{p}\left\langle J_{\lambda}^{\prime}\left(u_{k}, v_{k}\right),\left(u_{k}, v_{k}\right)\right\rangle\right]=c.
$$
As a consequence, we get that
\begin{equation}\label{p6}
\lim\limits_{k \rightarrow \infty}\left\|\left(u_{k}, v_{k}\right)\right\|_{\lambda}^{p}=\lim\limits_{k \rightarrow \infty} \int_{V} (R_\alpha\ast F(u_k,v_k))F(u_{k}, v_{k}) \mathrm{~d} \mu=\frac{2\gamma p c}{2\gamma-p}.
\end{equation}

For any $(u, v) \in W_\lambda$, by (\ref{p1}), we get that
$$
\begin{aligned}
\left\langle J_{\lambda}^{\prime}(u, v),(u, v)\right\rangle & =\|(u, v)\|_{\lambda}^{p}-\int_{V}(R_\alpha\ast F(u,v)) F(u, v) \mathrm{~d} \mu \\
& \geq\|(u, v)\|_{\lambda}^{p}-C_{N,\alpha,M_F}\|(u,v)\|^{2\gamma}_{\lambda}.
\end{aligned}
$$
Note that $2\gamma>p$. Let $\rho=\left(\frac{3}{4 C_{N,\alpha,M_{F} }}\right)^{\frac{1}{2\gamma-p}}>0$. If $\|(u, v)\|_{\lambda} \leq \rho$, then one gets that
$$
\left\langle J_{\lambda}^{\prime}(u, v),(u, v)\right\rangle \geq \frac{1}{4}\|(u, v)\|_{\lambda}^{p}.
$$

Let $c_0=\frac{(2\gamma-p) \rho^{p}}{2\gamma p}$. If $c<c_0$, we prove that  $c=0.$  In fact, by (\ref{p6}), one has that
$$
\lim _{k \rightarrow \infty}\left\|\left(u_{k}, v_{k}\right)\right\|_{\lambda}^{p}=\frac{2\gamma p c}{2\gamma-p}<\rho^{p}.
$$
Thus $\left\|\left(u_{k}, v_{k}\right)\right\|_{\lambda}\leq \rho$ for $k$ large enough. The above arguments yield that
$$
o_{k}(1)\left\|\left(u_{k}, v_{k}\right)\right\|_{\lambda}\geq\left\langle J_{\lambda}^{\prime}\left(u_{k}, v_{k}\right),\left(u_{k}, v_{k}\right)\right\rangle \geq \frac{1}{4}\left\|\left(u_{k}, v_{k}\right)\right\|_{\lambda}^{p},
$$
which implies $\left\|\left(u_{k}, v_{k}\right)\right\|_{\lambda} \rightarrow 0$, and hence $c=0$. 
\end{proof}

\begin{lm}\label{lm12}
Assume that $(F_1)$ and  $(A_1)$-$(A'_2)$ hold.  Let $c_*$ be a fixed constant. For any $\varepsilon>0$, there exist $\lambda_{\varepsilon}>0$ and $r_{\varepsilon}>0$ such that if $\{(u_k,v_k)\}\subset W_\lambda$ is a $(PS)_c$ sequence of the functional $J_\lambda$ with $c\leq c_*$ and $\lambda\geq\lambda_{\varepsilon}$, then we have that
\begin{eqnarray*}
\underset{k\rightarrow\infty}{\limsup}\int_{V\backslash B_{r_\varepsilon}} \left(R_\alpha\ast F(u_k,v_k)\right)F(u_k, v_k) \mathrm{~d} \mu\leq\varepsilon.
\end{eqnarray*}

\end{lm}

\begin{proof}
For $r\geq1$, let
\begin{eqnarray*}
\Omega^{+}_{r}=\{x\in V: |x|>r,~ a(x)\geq M\},\qquad \text{and}\qquad \Omega^{-}_{r}=\{x\in V: |x|> r,~ a(x)< M\}.
\end{eqnarray*}
 By (ii) of Lemma \ref{pb}, one has that
\begin{eqnarray*}
\int_{\Omega^{+}_{r}}|u_k|^p\mathrm{~d}\mu&\leq&\frac{1}{1+\lambda M}\int_{V}(1+\lambda a(x))|u_k|^p\mathrm{~d}\mu\\&\leq&\frac{1}{1+\lambda M}\left(\frac{2\gamma pc_*}{2\gamma-p}+o_k(1)\right)
\\&\rightarrow&0, \qquad\lambda\rightarrow\infty.
\end{eqnarray*}
For $q>1$, by the H\"{o}lder inequality and ($A'_2$), for $k$ large enough, we get that
\begin{eqnarray*}
\int_{\Omega^{-}_{r}}|u_k|^p\mathrm{~d}\mu&\leq&\left(\int_{V}|u_k|^{pq}\mathrm{~d}\mu\right)^{\frac{1}{q}}\left(\mu(\Omega^{-}_{r})\right)^{1-\frac{1}{q}}\\&=&\|u_k\|^{p}_{pq}\left(\mu(\Omega^{-}_{r})\right)^{1-\frac{1}{q}}\\&\leq&\|u_k\|^p_{p}\left(\mu(\Omega^{-}_{r})\right)^{1-\frac{1}{q}}\\ &\leq&\|(u_k,v_k)\|^p_{\lambda}\left(\mu(\Omega^{-}_{r})\right)^{1-\frac{1}{q}}\\ &\leq &\left(\frac{2pc_*}{p-1}+1\right)\left(\mu(\Omega^{-}_{r})\right)^{1-\frac{1}{q}}\\&\rightarrow&0, \qquad r\rightarrow\infty.
\end{eqnarray*}
Then we get that
$$\int_{V\backslash B_r}|u_k|^p\mathrm{~d}\mu=\int_{\Omega^{+}_{r}}|u_k|^p\mathrm{~d}\mu+\int_{\Omega^{-}_{r}}|u_k|^p\mathrm{~d}\mu\rightarrow 0, \qquad \lambda,~r\rightarrow\infty.$$
Similarly, we have
$$\int_{V\backslash B_r}|v_k|^p\mathrm{~d}\mu=\int_{\Omega^{+}_{r}}|v_k|^p\mathrm{~d}\mu+\int_{\Omega^{-}_{r}}|v_k|^p\mathrm{~d}\mu\rightarrow 0, \qquad \lambda,~r\rightarrow\infty.$$

Let $\phi\in C(V)$ such that $\phi(x)=1$ for $|x|> r$ and $\phi(x)=0$ for $|x|\leq r$. Similar to (\ref{p1}), we have that

$$\begin{aligned}
     &\int_{V\backslash B_{r}} \left(R_\alpha\ast F(u_k,v_k)\right)F(u_k, v_k) \mathrm{~d} \mu\\=&\int_{V} \left(R_\alpha\ast F(u_k,v_k)\right)F(\phi(u_k, v_k)) \mathrm{~d} \mu\\ \leq & C_{N,\alpha}\left(\int_V|F(u_k,v_k)|^{\frac{2N}{N+\alpha}}\mathrm{~d}\mu\right)^{\frac{N+\alpha}{2N}} \left(\int_V|F(\phi(u_k,v_k))|^{\frac{2N}{N+\alpha}}\mathrm{~d}\mu\right)^{\frac{N+\alpha}{2N}}\\ \leq& C_{N,\alpha,M_F}\left(\int_V(|u_k|^{\gamma}+|v_k|^{\gamma})^{\frac{2N}{N+\alpha}}\mathrm{~d}\mu\right)^{\frac{N+\alpha}{2N}}\left(\int_V(|\phi u_k|^{\gamma}+|\phi v_k|^{\gamma})^{\frac{2N}{N+\alpha}}\mathrm{~d}\mu\right)^{\frac{N+\alpha}{2N}}\\ \leq& C_{N,\alpha,M_F}\left(\int_V\left(|u_k|^{\frac{2N\gamma}{N+\alpha}}+|v_k|^{\frac{2N\gamma}{N+\alpha}}\right)\mathrm{~d}\mu\right)^{\frac{N+\alpha}{2N}}\left(\int_V\left(|\phi u_k|^{\frac{2N\gamma}{N+\alpha}}+|\phi v_k|^{\frac{2N\gamma}{N+\alpha}}\right)\mathrm{~d}\mu\right)^{\frac{N+\alpha}{2N}}\\ \leq &C_{N,\alpha,M_F}\left(\|u_k\|^{\gamma}_{p}+\|v_k\|^{\gamma}_{p}\right)\left(\|u_k\|^{\gamma}_{\ell^{p}(V\backslash B_{r})}+\|v_k\|^{\gamma}_{\ell^{p}(V\backslash B_{r})}\right)\\ \leq & C_{N,\alpha,M_F}\|(u_k,v_k)\|^{\gamma}_{\lambda}\left(\|u_k\|^{\gamma}_{\ell^{p}(V\backslash B_{r})}+\|v_k\|^{\gamma}_{\ell^{p}(V\backslash B_{r})}\right)\\ \leq & C_{N,\alpha,M_F}(\frac{2\gamma p c_*}{2\gamma-p}+1)^{\frac{\gamma}{p}}\left(\|u_k\|^{\gamma}_{\ell^{p}(V\backslash B_{r})}+\|v_k\|^{\gamma}_{\ell^{p}(V\backslash B_{r})}\right)\\ \rightarrow &0,\qquad\lambda,~r\rightarrow \infty.
\end{aligned}$$

\end{proof}

\begin{lm}\label{lm10}
 Assume that $(F_1)$ and  $(A_1)$-$(A'_2)$ hold. Let $\{(u_k,v_k)\}\subset W_\lambda$ be a $(PS)_c$ sequence of the functional $J_\lambda$. Passing to a subsequence if necessary, there exists $(u,v)\in W_\lambda$ such that
\begin{itemize}
\item[(i)] $(u_k,v_k)\rightharpoonup (u,v),\qquad \text{in}~W_{\lambda}$;
\item[(ii)] $(u_k,v_k)\rightarrow (u,v),\qquad \text{pointwise~in}~V$;
\item[(iii)] $J'_\lambda(u,v)=0,\qquad \text{in}~W^{*}_{\lambda}$.
\end{itemize}
\end{lm}

\begin{proof}
(i)
By Lemma \ref{pb} (i), one gets that $\{(u_k,v_k)\}$ is bounded in $W_\lambda$.
Then up to a subsequence, there exists $(u,v)\in W_\lambda$ such that $(u_k,v_k)\rightharpoonup (u,v)$ as $k\rightarrow\infty$.

(ii) Clearly, $\{(u_k,v_k)\}\subset W_{\lambda}$ is bounded in $\ell^p(V,\mathbb{R}^2)$, and hence bounded in $\ell^{\infty}(V,\mathbb{R}^2)$. Therefore, by diagonal principle, there exists a subsequence of $\{(u_k,v_k)\}$  pointwise converging to $(u,v)$.

(iii) We only neeed to prove that for any $(\phi,\psi)\in C_{c}(V)\times C_c(V)$, $\langle J'_\lambda(u,v),(\phi,\psi)\rangle=0.$
For any $\phi\in C_{c}(V)$, assume that $\text{supp}(\phi)\subseteq B_{r}$ with $r>1$. Since $B_{r+1}\subset V$ is a finite set and $u_k\rightarrow u $ pointwise in $V$ as $k\rightarrow\infty$, for any $s\geq 1$,
\begin{equation}\label{84}
u_k\rightarrow u,\qquad \text{in}~\ell^{s}(B_{r+1}).
\end{equation}
Note that
$$\begin{aligned}
 &\langle J'_\lambda(u_k,v_k),(\phi,\psi)\rangle \\=&
\int_{V}|\nabla u_k|^{p-2} \nabla u_k \nabla \phi\mathrm{~d}\mu+\int_{V}(\lambda a+1)|u_k|^{p-2} u_k \phi \mathrm{~d} \mu-\frac{1}{\gamma}\int_{V} \left(R_\alpha\ast F(u_k,v_k)\right) F_{u}(u_k, v_k)\phi\mathrm{~d}\mu\\ &+\int_{V}|\nabla v_k|^{p-2} \nabla v_k \nabla \psi \mathrm{~d} \mu +\int_{V}(\lambda b+1)|v_k|^{p-2} v_k \psi \mathrm{~d} \mu-\frac{1}{\gamma}\int_{V} \left(R_\alpha\ast F(u_k,v_k)\right) F_{v}(u_k, v_k)\psi\mathrm{~d}\mu\\=&-\int_{V}\left(\Delta_{p} u_k\right) \phi \mathrm{~d} \mu+\int_{V}(\lambda a+1)|u_k|^{p-2} u_k \phi\mathrm{~d} \mu-\frac{1}{\gamma}\int_{V} \left(R_\alpha\ast F(u_k,v_k)\right) F_{u}(u_k, v_k)\phi\mathrm{~d}\mu\\ &-\int_{V}\left(\Delta_{p} v_k\right) \psi\mathrm{~d} \mu +\int_{V}(\lambda b+1)|v_k|^{p-2} v_k \psi\mathrm{~d} \mu-\frac{1}{\gamma}\int_{V} \left(R_\alpha\ast F(u_k,v_k)\right) F_{v}(u_k, v_k)\psi\mathrm{~d}\mu.
\end{aligned}$$
Since $u_k\rightarrow u $ pointwise in $V$, we have
$\Delta_{p} u_k\rightarrow \Delta_{p} u$ pointwise in $V$.
Hence we get that
$$\lim\limits_{k\rightarrow\infty}\int_{V}\left(\Delta_{p} u_k-\Delta_{p} u\right)\phi\mathrm{~d}\mu=\lim\limits_{k\rightarrow\infty}\int_{B_{r+1}}\left(\Delta_{p} u_k-\Delta_{p} u\right)\phi\mathrm{~d}\mu=0,$$
and
$$\lim\limits_{k\rightarrow\infty}\int_{V}(\lambda a+1)\left(|u_k|^{p-2} u_k-|u|^{p-2} u\right) \phi \mathrm{~d} \mu=\lim\limits_{k\rightarrow\infty}\int_{B_r}(\lambda a+1)\left(|u_k|^{p-2} u_k-|u|^{p-2} u\right) \phi \mathrm{~d} \mu=0.$$
In the following, we prove that
\begin{equation}\label{p8}
\lim\limits_{k\rightarrow\infty}\int_{V} \left(R_\alpha\ast F(u_k,v_k)\right) F_{u}(u_k, v_k)\phi\mathrm{~d}\mu=\int_{V} \left(R_\alpha\ast F(u,v)\right) F_{u}(u, v)\phi\mathrm{~d}\mu.  
\end{equation}
In fact, a direct calculation yields that
$$\begin{aligned}
   & \left|\int_{V} \left[\left(R_\alpha\ast F(u_k,v_k)\right) F_{u}(u_k, v_k)-\left(R_\alpha\ast F(u,v)\right) F_{u}(u, v)\right]\phi\mathrm{~d}\mu\right|\\ \leq & \left|\int_{V} \left(R_\alpha\ast F(u_k,v_k)\right) F_{u}(u, v)\phi\mathrm{~d}\mu-\int_{V} \left(R_\alpha\ast F(u,v)\right) F_{u}(u, v)\phi\mathrm{~d}\mu\right|\\&+\left|\int_{V} \left(R_\alpha\ast F(u_k,v_k)\right)\left(F_{u}(u_k, v_k)- F_{u}(u, v)\right)\phi\mathrm{~d}\mu\right|\\ =:& J_1+J_2.
\end{aligned}$$

We first prove that $J_1\rightarrow 0$ as $k\rightarrow\infty$. By (\ref{04}), (\ref{p}) and the boundedness of $\left\{\left(u_{k}, v_{k}\right)\right\}$ in $W_\lambda$, we get that
$$\|R_\alpha\ast F(u_k, v_k)\|_{\frac{2N}{N-\alpha}}\leq C_{N,\alpha}\|F(u_k,v_k)\|_{\frac{2N}{N+\alpha}}\leq C_{N,\alpha,M_F}\|(u_k,v_k)\|^{\gamma}_{\lambda}<\infty.$$ Thus, up to a subsequence, we have that
$R_\alpha\ast F(u_k,v_k)\rightharpoonup R_\alpha\ast F(u,v)$ in $\ell^{\frac{2N}{N-\alpha}}(V,\mathbb{R}^2)$. Note that $F_u(u,v)\phi\in \ell^{\frac{2N}{N+\alpha}}(V,\mathbb{R}^2)$, then $J_1\rightarrow0$ as $k\rightarrow\infty$.

Next, we prove that $J_2\rightarrow 0$ as $k\rightarrow\infty$. By (ii), we have that $F_u(u_k,v_k)\rightarrow F_u(u,v)$ pointwise in $V$. Then combined with the boundedness of $\left\{\left(u_{k}, v_{k}\right)\right\}$ in $W_\lambda$, we obtain that
$$
\begin{aligned}
 J_2\leq & C_{N,\alpha}\left(\int_V|F(u_k,v_k)|^{\frac{2N}{N+\alpha}}\mathrm{~d}\mu\right)^{\frac{N+\alpha}{2N}} \left(\int_V\left|(F_u(u_k,v_k)-F_u(u, v))\phi\right|^{\frac{2N}{N+\alpha}}\mathrm{~d}\mu\right)^{\frac{N+\alpha}{2N}} \\ \leq &C_{N,\alpha,M_F}\|(u_k,v_k)\|^{\gamma}_{\lambda}\left(\int_{B_r}\left|(F_u(u_k,v_k)-F_u(u, v))\phi\right|^{\frac{2N}{N+\alpha}}\mathrm{~d}\mu\right)^{\frac{N+\alpha}{2N}}\\ \rightarrow & 0.
\end{aligned}
$$
Therefore, we prove that (\ref{p8}) holds. 

By similar arguments, we have that
$$\lim\limits_{k\rightarrow\infty}\int_{V}\left(\Delta_{p} v_k-\Delta_{p} v\right)\psi\mathrm{~d}\mu=\lim\limits_{k\rightarrow\infty}\int_{B_{r+1}}\left(\Delta_{p} v_k-\Delta_{p} v\right)\psi\mathrm{~d}\mu=0,$$
$$\lim\limits_{k\rightarrow\infty}\int_{V}(\lambda b+1)\left(|v_k|^{p-2} v_k-|v|^{p-2} v\right) \psi \mathrm{~d} \mu=\lim\limits_{k\rightarrow\infty}\int_{B_r}(\lambda b+1)\left(|v_k|^{p-2} v_k-|v|^{p-2} v\right) \psi \mathrm{~d} \mu=0,$$
and $$
\lim\limits_{k\rightarrow\infty}\int_{V} \left(R_\alpha\ast F(u_k,v_k)\right) F_{v}(u_k, v_k)\psi\mathrm{~d}\mu=\int_{V} \left(R_\alpha\ast F(u,v)\right) F_{v}(u, v)\psi\mathrm{~d}\mu.  
$$
The above arguments yield that
$$\langle J'_\lambda(u,v),(\phi,\psi)\rangle=\lim\limits_{k\rightarrow\infty}\langle J'_\lambda(u_k,v_k),(\phi,\psi)\rangle=0.$$
\end{proof}

\begin{lm}\label{lm11}
Assume that $(F_1)$ and  $(A_1)$-$(A'_2)$ hold. Let $\{(u_k,v_k)\}\subset W_\lambda$ be a $(PS)_c$ sequence of the functional $J_\lambda$. Up to a subsequence, there exists $(u,v)\in W_\lambda$ such that
\begin{itemize}
\item[(i)] $\underset{k\rightarrow\infty}{\lim}J_\lambda(u_k-u,v_k-v)=c-J_\lambda(u,v)$;
\item[(ii)] $\underset{k\rightarrow\infty}{\lim}J'_\lambda(u_k-u,v_k-v)=0,\qquad \text{in~}W^{*}_\lambda$.
\end{itemize}

\end{lm}

\begin{proof}
It follows from Lemma \ref{lm10} that
\begin{eqnarray*}
\begin{array}{ll}
\|(u_k,v_k)\|_{\lambda}\leq C,\qquad \text{and}\qquad
(u_k,v_k)\rightarrow (u,v), \qquad \text{pointwise~in}~V.
\end{array}
\end{eqnarray*}

(i) By the Br\'{e}zis-Lieb lemma, we have that
$$
\int_{V}(\lambda a+1)|u_k|^p\mathrm{~d}\mu-\int_{V}(\lambda a+1)|u_k-u|^p\mathrm{~d}\mu=\int_{V}(\lambda a+1)|u|^p \mathrm{~d}\mu+o_k(1),
$$
and 
$$
\int_{V}(\lambda b+1)|v_k|^p\mathrm{~d}\mu-\int_{V}(\lambda b+1)|v_k-v|^p\mathrm{~d}\mu=\int_{V}(\lambda b+1)|v|^p \mathrm{~d}\mu+o_k(1).
$$

By Corollary 11 in \cite{HL}, one gets that
$$
\int_{V}|\nabla u_k|^p\mathrm{~d}\mu-\int_{V}|\nabla (u_k-u)|^p \mathrm{~d}\mu=\int_{V}|\nabla u|^p \mathrm{~d}\mu+o_k(1),
$$
and
$$
\int_{V}|\nabla v_k|^p\mathrm{~d}\mu-\int_{V}|\nabla (v_k-v)|^p \mathrm{~d}\mu=\int_{V}|\nabla v|^p \mathrm{~d}\mu+o_k(1).
$$
Then combined with Lemma \ref{pa}, one gets that
\begin{eqnarray*}
J_\lambda(u_k,v_k)-J_\lambda(u_k-u,v_k-v)=J_\lambda(u,v)+o_k(1).
\end{eqnarray*}
Note that $\underset{k\rightarrow\infty}{\lim}J_\lambda(u_k,v_k)=c$, then we obtain that
\begin{eqnarray*}
J_\lambda(u_k-u,v_k-v)=c-J_\lambda(u,v)+o_k(1).
\end{eqnarray*}

(ii)
For any $\phi\in C_{c}(V)$, let $\text{supp}(\phi)\subseteq B_{r}$ with $r>1$.  Denote 
$$w_k=u_k-u,\qquad z_k=v_k-v.$$
Then we have that $w_k\rightarrow0$ and $z_k\rightarrow 0$ pointwise in $V$.

For any $(\phi,\psi)\in C_c(V)\times C_c(V)$, we have that 
$$\begin{aligned}
 &\langle J'_\lambda(w_k,z_k),(\phi,\psi)\rangle \\=&
\int_{V}|\nabla w_k|^{p-2} \nabla w_k \nabla \phi\mathrm{~d}\mu+\int_{V}(\lambda a+1)|w_k|^{p-2} w_k \phi \mathrm{~d} \mu-\frac{1}{\gamma}\int_{V} \left(R_\alpha\ast F(w_k,z_k)\right) F_{u}(w_k, z_k)\phi\mathrm{~d}\mu\\ &+\int_{V}|\nabla z_k|^{p-2} \nabla z_k \nabla \psi \mathrm{~d} \mu +\int_{V}(\lambda b+1)|z_k|^{p-2} z_k \psi \mathrm{~d} \mu-\frac{1}{\gamma}\int_{V} \left(R_\alpha\ast F(w_k,z_k)\right) F_{v}(w_k, z_k)\psi\mathrm{~d}\mu.
\end{aligned}$$
By the H\"{o}lder inequality, we obtain that
$$\begin{aligned}
\left|\int_{V}\left(|\nabla w_k|^{p-2} \nabla w_k \nabla \phi\right)\mathrm{~d}\mu\right|\leq&\int_{B_{r+1}}|\nabla w_k|^{p-1} |\nabla \phi|\mathrm{~d}\mu\\ \leq &  \|\nabla w_k\|_{\ell^{p}(B_{r+1})}^{p-1} \|\nabla\phi\|_{\ell^{p}(B_{r+1})}\\\leq & o_k(1)\|(\phi,\psi)\|_{\lambda},
\end{aligned}$$
and
$$\begin{aligned}
    \left|\int_{V}(\lambda a+1)|w_k|^{p-2} w_k \phi\mathrm{~d} \mu\right|  \leq &\int_{B_r}(\lambda a+1)|w_k|^{p-1} |\phi|\mathrm{~d} \mu\\ \leq &\left(\int_{B_r} (\lambda a+1)|w_k|^{p}\mathrm{~d} \mu\right)^{\frac{p-1}{p}}\left(\int_{B_r} (\lambda a+1)|\phi|^{p}\mathrm{~d} \mu\right)^{\frac{1}{p}}\\ \leq & o_k(1)\|(\phi,\psi)\|_{\lambda}.
\end{aligned}$$

By the HLS inequality (\ref{bo}), the boundedness of $\{(w_k,z_k)\}$ in $W_\lambda$, $(w_k,z_k)\rightarrow (0,0)$ pointwise in $V$ and the H\"{o}lder inequality, one has that
$$\begin{aligned}
    &\left|\int_{V} \left(R_\alpha\ast F(w_k,z_k)\right) F_{u}(w_k, z_k)\phi\mathrm{~d}\mu\right|\\ \leq &C_{N,\alpha}\left(\int_V|F(w_k,z_k)|^{\frac{2N}{N+\alpha}}\mathrm{~d}\mu\right)^{\frac{N+\alpha}{2N}} \left(\int_V\left|F_u(w_k,z_k)\phi\right|^{\frac{2N}{N+\alpha}}\mathrm{~d}\mu\right)^{\frac{N+\alpha}{2N}} \\ \leq &C_{N,\alpha,M_F}\|(w_k,z_k)\|^{\gamma}_{\lambda}\left(\int_{B_r}|F_u(w_k,z_k)|^{\frac{2N}{N+\alpha}}|\phi|^{\frac{2N}{N+\alpha}}\mathrm{~d}\mu\right)^{\frac{N+\alpha}{2N}}\\ \leq &C_{N,\alpha,M_F} \left[\int_{B_r}\left(|w_k|^{\gamma-1}+|z_k|^{\gamma-1}\right)^{\frac{2N}{N+\alpha}}|\phi|^{\frac{2N}{N+\alpha}}\mathrm{~d}\mu \right]^{\frac{N+\alpha}{2N}}\\ \leq &C_{N,\alpha,M_F} \left[\int_{B_r}\left(|w_k|^{^{\frac{2N(\gamma-1)}{N+\alpha}}}+|z_k|^{^{\frac{2N(\gamma-1)}{N+\alpha}}}\right)|\phi|^{\frac{2N}{N+\alpha}}\mathrm{~d}\mu \right]^{\frac{N+\alpha}{2N}}\\ \leq & C_{N,\alpha,M_F} \left[\int_{B_r}\left(|w_k|^{^{\frac{2N\gamma}{N+\alpha}}}+|z_k|^{^{\frac{2N\gamma}{N+\alpha}}}\right)\mathrm{~d}\mu\right]^{\frac{(N+\alpha)(\gamma-1)}{2N\gamma}}\left(\int_{B_r}|\phi|^{\frac{2N\gamma}{N+\alpha}}\mathrm{~d}\mu\right)^{\frac{(N+\alpha)}{2N\gamma}}\\ \leq &C_{N,\alpha,M_F} \left[\int_{B_r}\left(|w_k|^{^{\frac{2N\gamma}{N+\alpha}}}+|z_k|^{^{\frac{2N\gamma}{N+\alpha}}}\right)\mathrm{~d}\mu\right]^{\frac{(N+\alpha)(\gamma-1)}{2N\gamma}}\|\phi\|_p\\ \leq &o_k(1)\|(\phi,\psi)\|_{\lambda},
\end{aligned}$$
where we have used the fact that $\frac{2N\gamma}{N+\alpha}>p$ in the sixth inequality.

By similar arguments as above, we have that
$$\left|\int_{V}|\nabla z_k|^{p-2} \nabla z_k \nabla \psi\mathrm{~d}\mu\right|\leq o_k(1)\|(\phi,\psi)\|_{\lambda},$$
$$\left|\int_{V}(\lambda b+1)|z_k|^{p-2} z_k \psi \mathrm{~d} \mu\right|\leq  o_k(1)\|(\phi,\psi)\|_{\lambda},$$
and 
$$\left|\int_{V} \left(R_\alpha\ast F(w_k,z_k)\right) F_{v}(w_k, z_k)\psi\mathrm{~d}\mu\right|\leq o_k(1)\|(\phi,\psi)\|_{\lambda}.$$
Therefore, the above results imply that
$$\left|\langle J'_\lambda(w_k,z_k),(\phi,\psi)\rangle\right|\leq o_k(1)\|(\phi,\psi)\|_{\lambda}.$$
Then we get that
$$\underset{k\rightarrow\infty}{\lim}\|J'(w_k,z_k)\|_{W^{*}_\lambda}=\underset{k\rightarrow\infty}{\lim}
\underset{\|(\phi,\psi)\|_{\lambda}=1}{\sup}|\langle J'(w_k,z_k),(\phi,\psi)\rangle|=0.$$

\end{proof}

Then the above lemmas imply a compactness result.
\begin{lm}\label{lm13} 
Let $(F_1)$ and $(A_1)$-$(A'_2)$ hold. For any $c^*>0$, there exists $\lambda^*>0$ such that $J_\lambda$ satisfies $(PS)_c$ condition for all $c\leq c^*$ and $\lambda\geq\lambda^*$.
\end{lm}
\begin{proof}
Let $c_0$ be given by Lemma \ref{pb} (ii) and choose $\varepsilon>0$ such that $\varepsilon<\frac{2\gamma pc_0}{2\gamma-p}$. Then for the given $c^*>0$, we choose $\lambda^*=\lambda_{\varepsilon}>0$ and $r_\varepsilon>0$ in Lemma \ref{lm12}.

Let $\{(u_k,v_k)\}\subset W_\lambda$ be a $(PS)_c$ sequence of the functional $J_\lambda$ with $c\leq c^*$ and $\lambda\geq\lambda^*$. By Lemma \ref{lm10}, there exists $(u,v)\in W_\lambda$ such that
\begin{eqnarray*}
(u_k,v_k)\rightharpoonup (u,v),\qquad \text{in}~W_\lambda,\qquad\text{and}\qquad (u_k,v_k)\rightarrow (u,v),\qquad \text{pointwise~in}~V.
\end{eqnarray*}
Denote $$w_k=u_k-u,\quad z_k=v_k-v.$$
By Lemma \ref{lm11}, one sees that $\{(w_k,z_k)\}\subset W_\lambda$ is a $(PS)_d$ sequence of the functional $J_\lambda$ with $d=c-J_\lambda(u,v)$. We claim that $d=0$. By contradiction, if $d\neq0$, then by Lemma \ref{pb} (ii), $d\geq c_0>0$. Moreover, we have that
\begin{eqnarray*}
\underset{k\rightarrow\infty}{\lim}\int_{V}(R_\alpha\ast F(w_k,z_k))F(w_{k}, z_{k}) \mathrm{~d} \mu=\frac{2\gamma pd}{2\gamma-p}\geq\frac{2\gamma pc_0}{2\gamma-p}.
\end{eqnarray*}
On the other hand, by Lemma \ref{lm12}, we have that
\begin{eqnarray*}
\underset{k\rightarrow\infty}{\limsup}\int_{V\setminus B_{r_{\varepsilon}}} (R_\alpha\ast F(w_k,z_k))F(w_{k}, z_{k}) \mathrm{~d} \mu\leq\varepsilon<\frac{2\gamma pc_0}{2\gamma-p}.
\end{eqnarray*}
The above two inequalities imply that $(w_k,z_k)\rightarrow (w,z)$ pintwise in $V$ with some $(w,z)\neq (0,0)$, which contradicts $(w_k,z_k)\rightarrow (0,0)$ pointwise in $V$. Hence $d=0$. By Lemma \ref{pb} (ii) again, we get that
\begin{eqnarray*}
\underset{k\rightarrow\infty}{\lim}\|(w_k,z_k)\|^{p}_{W_\lambda}=\frac{2\gamma pd}{2\gamma-p}=0,
\end{eqnarray*}
which means that $(u_k,v_k)\rightarrow (u,v)$ in $W_\lambda$.    
\end{proof}

In the following, we prove the existence of ground state solutions  to the system (\ref{00}).

\
\

{\bf Proof of Theorem \ref{th1}:} By Lemma \ref{lm6}, one gets that $J_\lambda$ satisfies the mountain-pass geometry. Hence there exists a sequence $\{(u_k,v_k)\}\subset W_\lambda$ such that
\begin{eqnarray*}\label{3-9}
J_\lambda(u_k,v_k)\rightarrow m_\lambda, \qquad\text{in~}W_\lambda,\qquad\text{and}\qquad J'_\lambda(u_k,v_k)\rightarrow 0,\qquad\text{in~} W^{*}_\lambda.
\end{eqnarray*}
Then it follows from Lemma \ref{lm10} that there exists $(u_\lambda,v_\lambda)\in W_\lambda$ such that, up to a subsequence,
\begin{eqnarray*}\label{3-0}
\begin{array}{ll}
(u_k,v_k)\rightharpoonup (u_{\lambda},v_\lambda),\qquad \text{in}~W_\lambda,\\
(u_k,v_k)\rightarrow (u_{\lambda},v_\lambda), \qquad \text{pointwise~in}~V,\\
J'_{\lambda}(u_{\lambda},v_\lambda)=0,\qquad \text{in}~W^{*}_\lambda.
\end{array}
\end{eqnarray*}
By Lemma \ref{lm13}, there exists $\lambda_0>0$ such that, for any $\lambda\geq\lambda_0$, $(u_k,v_k)\rightarrow (u_\lambda,v_\lambda)$ in $W_\lambda$.
Then it follows from Lemma \ref{pb} (ii) and $m_\lambda>0$ that
\begin{eqnarray*}
\|(u_\lambda,v_\lambda)\|^{p}_{W_\lambda}=\underset{k\rightarrow\infty}{\lim}\|(u_k,v_k)\|^{p}_{W_\lambda}=\frac{2\gamma pm_\lambda}{2\gamma-p}>0,
\end{eqnarray*}
which yields $(u_\lambda,v_\lambda)\neq(0,0)$. Then $(u_\lambda,v_\lambda)\in \mathcal{N}_\lambda$. Moreover, we have that
\begin{align*}
J_\lambda(u_\lambda,v_\lambda)=&J_\lambda(u_\lambda,v_\lambda)-\frac{1}{2\gamma}\langle J'_\lambda(u_\lambda,v_\lambda),(u_\lambda,v_\lambda)\rangle\\=&(\frac{1}{p}-\frac{1}{2\gamma})\|(u_\lambda,v_\lambda)\|^p_{\lambda}\\=&\lim\limits_{k\rightarrow\infty}(\frac{1}{p}-\frac{1}{2\gamma})\|(u_k,v_k)\|^p_{\lambda}\\=&\underset{k\rightarrow\infty}{\lim}\left(J_\lambda(u_k,v_k)-\frac{1}{2\gamma}\langle J'_\lambda(u_k,v_k),(u_k,v_k)\rangle\right)\\=& m_\lambda>0.
\end{align*}
Therefore, $(u_\lambda,v_\lambda)\in \mathcal{N}_\lambda$ is a ground state solution of the system (\ref{00}).\qed

\section{The asymptotic behavior of ground state solutions}
In this section, we prove that the ground state solution $(u_{\lambda}, v_{\lambda})$ of the system (\ref{00}) converges to a ground state solution of  the system (\ref{01}) as $\lambda \rightarrow \infty$. The following result plays a key role in the proof of Theorem \ref{th4}.

\begin{lm}\label{pc}
Let $(F_{1})$ and $(A_{1})$-$(A'_{2})$ hold. Then we have that $\lim\limits_{\lambda \rightarrow \infty}m_{\lambda} =m_{\Omega}$.    
\end{lm}

\begin{proof}
For any $\lambda>0$, since $\mathcal{N}_{\Omega} \subset \mathcal{N}_{\lambda}$, we have that $m_{\lambda} \leq m_{\Omega}$. By contradiction, suppose that there exists a a sequence $\lambda_{k} \rightarrow \infty$ such that
\begin{equation}\label{f1}
\lim _{k \rightarrow \infty} m_{\lambda_{k}}=l< m_{\Omega}.
\end{equation}
By Theorem \ref{th1}, for $\lambda_{k}$ large enough, there exists a sequence $\{(u_{k}, v_{k})\}\subset\mathcal{N}_{\lambda_k}$, ground state solutions to the system (\ref{00}), such that $J_{\lambda_{k}}(u_{k}, v_{k})=m_{\lambda_{k}}>0$. This implies that $\left\{\left(u_{k}, v_{k}\right)\right\}$ is bounded in $W_\lambda$. Hence there exists $(u, v) \in W_\lambda$ such that
$$(u_{k}, v_{k}) \rightharpoonup(u, v),\quad \text{in~}W_\lambda,\qquad (u_{k}, v_{k}) \rightarrow(u, v),\quad \text{pointwise~in~}V.$$

We claim that $u \equiv 0$ in $\Omega_{a}^{c}$ and $v \equiv 0$ in $\Omega_{b}^{c}$. In fact, if there exists $x_{0} \in \Omega_{a}^{c}$ such that $u(x_0)\neq 0$, then
$$
\begin{aligned}
m_{\lambda_k}&=J_{\lambda_{k}}(u_{k}, v_{k})\\ & =J_{\lambda_{k}}(u_{k}, v_{k})-\frac{1}{2\gamma}\left\langle J_{\lambda_{k}}^{\prime}(u_{k}, v_{k}),(u_{k}, v_{k})\right\rangle \\
& =\left(\frac{1}{p}-\frac{1}{2\gamma}\right)\left\|\left(u_{k}, v_{k}\right)\right\|_{\lambda_{k}}^{p} \\
& \geq\left(\frac{1}{p}-\frac{1}{2\gamma}\right) \int_{V}\left(\lambda_{k} a+1\right)\left|u_{k}\right|^{p} \mathrm{~d} \mu \\
& \geq\left(\frac{1}{p}-\frac{1}{2\gamma}\right) \lambda_{k} a\left(x_{0}\right)\left|u_{k}\left(x_{0}\right)\right|^{p}\\ &\rightarrow \infty, \quad k \rightarrow \infty,
\end{aligned}
$$
which contradicts (\ref{f1}). Similarly, we also have $v \equiv 0$ in $\Omega_{b}^{c}$.

Now we show that $u_k\rightarrow u, v_k\rightarrow v$ in $\ell^q(V)$ for $q>p$. Otherwise, by Lemma \ref{lm16}, there exists $\delta>0$ such that $\underset{k\rightarrow\infty}{\lim}\|u_{k}-u\|_{\infty}=\delta>0$. Then there exists a sequences $\{x_{k}\}\subset V$ such that $\left|(u_{k}-u)(x_{k})\right|\geq\frac{\delta}{2}>0$. Since $(u_{k}-u)\rightarrow 0$ pointwise in $V$, 
we have that $|x_k|\rightarrow\infty$ as $k\rightarrow\infty$.

Note that $(u_{k},v_k)\in \mathcal{N}_{\lambda_k}$ and $\mu\left(B_{r}(x_k)\cap\{x: a(x)\leq M_1\}\right)\rightarrow 0$ as $k\rightarrow\infty$, where $0<r<1$. Then we have that
\begin{align*}
m_{\lambda_k}=&J_{\lambda_{k}}(u_{k}, v_{k})\\\geq&
\left(\frac{1}{p}-\frac{1}{2\gamma}\right)\int_{B_{r}(x_k)\cap\{x: a(x)\geq M_1\}}\lambda_k a |u_{k}-u|^2 {\rm ~d}\mu\\ \geq&\left(\frac{1}{p}-\frac{1}{2\gamma}\right)\lambda_k M_1\left(\int_{B_{r}(x_k)} |u_{k}-u|^2 {\rm ~d}\mu-\int_{B_{r}(x_k)\cap\{x: a(x)\leq M_1\}}|u_{k}-u|^2 {\rm ~d}\mu\right)\\\geq&
\left(\frac{1}{p}-\frac{1}{2\gamma}\right)\lambda_k M_1\left(\frac{\delta^2}{4}+o_k(1)\right)\\\rightarrow&\infty,\qquad k\rightarrow\infty.
\end{align*}
This is a contradiction. Hence for any $q>p$, $u_{k}\rightarrow u$ in $\ell^{q}(V).$ Similarly, we also have $v_{k}\rightarrow v$ in $\ell^{q}(V).$
Since $\frac{2N\gamma}{N+\alpha}>p$, we get that
$$
\begin{aligned}
 \int_{V} \left(R_\alpha\ast F(u_k-u,v_k-v)\right)F(u_k-u, v_k-v) \mathrm{~d} \mu \leq & C_{N,\alpha}\left(\int_V|F(u_k-u,v_k-v)|^{\frac{2N}{N+\alpha}}\mathrm{~d}\mu\right)^{\frac{N+\alpha}{N}} \\ \leq& C_{N,\alpha}\left(\int_V\left[M_{F}(|u_k-u|^{\gamma}+|v_k-v|^{\gamma})\right]^{\frac{2N}{N+\alpha}}\mathrm{~d}\mu\right)^{\frac{N+\alpha}{N}}\\ \leq& C_{N,\alpha,M_F}\left(\int_V\left(|u_k-u|^{\frac{2N\gamma}{N+\alpha}}+|v_k-v|^{\frac{2N\gamma}{N+\alpha}}\right)\mathrm{~d}\mu\right)^{\frac{N+\alpha}{N}}\\\leq &C_{N,\alpha,M_F}\left(\|u_k-u\|^{2\gamma}_{\frac{2N\gamma}{N+\alpha}}+\|v_k-v\|^{2\gamma}_{\frac{2N\gamma}{N+\alpha}}\right)\\ \rightarrow & 0,\qquad k\rightarrow\infty.
\end{aligned}
$$
Then it follows from Lemma \ref{pa} that
\begin{equation}\label{f2}
\lim\limits_{k\rightarrow\infty}\int_{V} \left(R_\alpha\ast F(u_k,v_k)\right)F(u_k, v_k) \mathrm{~d} \mu=\int_{V} \left(R_\alpha\ast F(u,v)\right)F(u, v) \mathrm{~d} \mu.
\end{equation}
This implies that $(u,v)\not\equiv (0,0)$ in $V$. Indeed, if $(u,v)\equiv(0,0)$ in $V$, since $(u_k,v_k)\in \mathcal{N}_{\lambda_k}$, by Lemma \ref{k9} (i), we get that
\begin{eqnarray*}
0=\lim\limits_{k\rightarrow\infty}\int_{V} \left(R_\alpha\ast F(u_k,v_k)\right)F(u_k, v_k) \mathrm{~d} \mu=\underset {k\rightarrow\infty}{\lim}
\left\|\left(u_{k}, v_{k}\right)\right\|_{\lambda_{k}}^{p}\geq\sigma^2>0,
\end{eqnarray*}
which is a contradiction. Hence $(u,v)\neq(0,0)$ in $\Omega\times\Omega$.

By the facts $u\equiv0 $ in $\Omega^c_a$, $v\equiv0 $ in $\Omega^c_b$ and the Fatou lemma, one has that
$$\begin{aligned}
\int_{\bar{\Omega}_{a} \cup \bar{\Omega}_{b}}\left(|\nabla u|^{p}+|\nabla v|^{p}\right) \mathrm{~d} \mu + \int_{\Omega_{a} \cup \Omega_{b}}\left(|u|^{p}+|v|^{p}\right) \mathrm{~d} \mu= &\int_V\left(|\nabla u|^p+|u|^p\right) \mathrm{~d}\mu+\int_V\left(|\nabla v|^p+|v|^p\right) \mathrm{~d}\mu \\ \leq &\liminf\limits_{k\rightarrow\infty}\int_V\left(|\nabla u_k|^p+(\lambda_ka+1)|u_k|^p\right) \mathrm{~d}\mu\\&+ \liminf\limits_{k\rightarrow\infty}\int_V\left(|\nabla v_k|^p+(\lambda_k b+1)|v_k|^p\right) \mathrm{~d}\mu\\=&\liminf\limits_{k\rightarrow\infty}\int_{V} \left(R_\alpha\ast F(u_k,v_k)\right)F(u_k, v_k) \mathrm{~d} \mu\\=&\int_{V} \left(R_\alpha\ast F(u,v)\right)F(u, v) \mathrm{~d} \mu\\=&\int_{\Omega_{a} \cup \Omega_{b}} \left(R_\alpha\ast F(u,v)\right)F(u, v) \mathrm{~d} \mu.
\end{aligned}$$
By Lemma \ref{l9}, there exists $t\in(0,1]$ such that $t(u, v) \in \mathcal{N}_{\Omega}$. Then combined with (\ref{f2}), we get that
$$
\begin{aligned}
J_{\Omega}(t u, t v)= & \left(\frac{1}{p}-\frac{1}{2\gamma}\right)\int_{\Omega_{a} \cup \Omega_{b}} \left(R_\alpha\ast F(tu,tv)\right)F(tu, tv) \mathrm{~d} \mu\\=&\left(\frac{1}{p}-\frac{1}{2\gamma}\right)\int_{V} \left(R_\alpha\ast F(tu,tv)\right)F(tu, tv) \mathrm{~d} \mu\\=&t^{2\gamma}\left(\frac{1}{p}-\frac{1}{2\gamma}\right)\int_{V} \left(R_\alpha\ast F(u,v)\right)F(u, v) \mathrm{~d} \mu\\= & t^{2\gamma}\lim_{k \rightarrow \infty}\left(\frac{1}{p}-\frac{1}{2\gamma}\right)\int_{V} \left(R_\alpha\ast F(u_k,v_k)\right)F(u_k, v_k) \mathrm{~d} \mu\\
=& t^{2\gamma}\lim_{k \rightarrow \infty} J_{\lambda_{k}}\left(u_{k}, v_{k}\right)\\\leq &l.
\end{aligned}
$$
Then we get that $m_\Omega\leq J_{\Omega}(tu)\leq l<m_\Omega.$ This is a contradiction. Hence 
$$
\lim _{\lambda \rightarrow \infty} m_{\lambda}=m_{\Omega}.
$$

\end{proof}

{\bf Proof of Theorem \ref{th4}.}
 We need to prove that for any sequence $\lambda_{k} \rightarrow \infty$, the corresponding ground state solution $(u_{k}, v_{k}) \in \mathcal{N}_{\lambda_{k}}$ satisfying $J_{\lambda_{k}}(u_{k}, v_{k})=m_{\lambda_{k}}$ converges in $W^{1,p}(V)\times W^{1,p}(V)$ to a ground state solution $(u, v)$ of the system (\ref{01}) along a subsequence. Since $\left\{\left(u_{\lambda_{k}}, v_{\lambda_{k}}\right)\right\}$ is bounded in $W^{1, p}(V) \times W^{1, p}(V)$, there exists $(u,v)\in  W^{1, p}(V) \times W^{1, p}(V)$ such that
$$(u_{k}, v_{k}) \rightharpoonup(u, v),\quad \text{in~}W^{1,p}(V)\times W^{1,p}(V),\qquad (u_{k}, v_{k}) \rightarrow(u, v),\quad \text{pointwise~in~}V.$$
Moreover, we get, from the proof of Lemma \ref{pb}, that $\left.u\right|_{\Omega_{a}^{c}} \equiv 0$, $\left.v\right|_{\Omega_{b}^{c}} \equiv 0$ and $u_k\rightarrow u, v_k\rightarrow v$ in $\ell^q(V)$ for $q>p$ with $(u,v)\neq 0$ in $\Omega\times\Omega$, and thus 
\begin{equation}\label{f4}
\lim\limits_{k\rightarrow\infty}\int_{V} \left(R_\alpha\ast F(u_k,v_k)\right)F(u_k, v_k) \mathrm{~d} \mu=\int_{V} \left(R_\alpha\ast F(u,v)\right)F(u, v) \mathrm{~d} \mu.
\end{equation}

Now, we prove that $(u, v)$ is a ground state solution to the system (\ref{01}). In fact, since $J_{\lambda_{k}}^{\prime}(u_{k}, v_{k})=0$, for any $\phi \in C_{c}(\Omega_{a})$, by using $(\phi, 0)$ as a test function, we have $\left\langle J_{\lambda_{k}}^{\prime}(u_{k}, v_{k}),(\phi, 0)\right\rangle=0$. Namely,
$$
\int_{V}\left(\left|\nabla u_{k}\right|^{p-2} \nabla u_{k} \nabla \phi+\left(\lambda_{k} a+1\right)\left|u_{k}\right|^{p-2} u_{k} \phi\right) \mathrm{~d} \mu=\frac{1}{\gamma} \int_{V} \left(R_\alpha\ast F(u_k,v_k)\right) F_{u}(u_k, v_k)\phi\mathrm{~d}\mu.
$$
Since $\phi=0$ on $\Omega_{a}^{c}$ and $a(x)=0$ on $\Omega_{a}$, we obtain that
$$\int_{\bar{\Omega}_{a}}\left|\nabla u_{k}\right|^{p-2} \nabla u_{k} \nabla \phi \mathrm{~d} \mu+\int_{\Omega_{a}}\left|u_{k}\right|^{p-2} u_{k} \phi \mathrm{~d} \mu=\frac{1}{\gamma} \int_{\Omega_{a}} \left(R_\alpha\ast F(u_k,v_k)\right) F_{u}(u_k, v_k)\phi\mathrm{~d}\mu.$$
Thus, as $k \rightarrow \infty$, the above equality turns to
$$
\int_{\bar{\Omega}_{a}}|\nabla u|^{p-2} \nabla u \nabla \phi \mathrm{~d} \mu+\int_{\Omega_{a}}|u|^{p-2} u \phi \mathrm{~d} \mu=\frac{1}{\gamma} \int_{\Omega_{a}} \left(R_\alpha\ast F(u,v)\right) F_{u}(u, v)\phi\mathrm{~d}\mu.
$$
Note that $|\nabla \phi|=0$ on $(\bar{\Omega}_{a})^{c}$, then we have
\begin{equation}\label{f5}
\int_{\bar{\Omega}_{a} \cup \bar{\Omega}_{b}}|\nabla u|^{p-2} \nabla u \nabla \phi \mathrm{~d} \mu+\int_{\Omega_{a} \cup \Omega_{b}}|u|^{p-2} u \phi \mathrm{~d} \mu=\frac{1}{\gamma} \int_{\Omega_{a} \cup \Omega_{b}} \left(R_\alpha\ast F(u,v)\right) F_{u}(u, v)\phi\mathrm{~d}\mu. 
\end{equation}

Similarly, for $\psi \in C_{c}(\Omega_{b})$, by using $(0, \psi)$ as a test function, we get
\begin{equation}\label{f6}
\int_{\bar{\Omega}_{a} \cup \bar{\Omega}_{b}}|\nabla v|^{p-2} \nabla v \nabla \psi \mathrm{~d} \mu+\int_{\Omega_{a} \cup \Omega_{b}}|v|^{p-2} v \psi \mathrm{~d} \mu=\frac{1}{\gamma} \int_{\Omega_{a} \cup \Omega_{b}} \left(R_\alpha\ast F(u,v)\right) F_{v}(u, v)\psi\mathrm{~d}\mu. 
\end{equation}
Then it follows from (\ref{f5}) and (\ref{f6}) that for any $(\phi, \psi) \in C_{c}\left(\Omega_{a}\right) \times C_{c}\left(\Omega_{b}\right)$, $\left\langle J_{\Omega}^{\prime}(u, v),(\phi, \psi)\right\rangle=0$. Hence, $(u, v)$ is a nontrivial critical point of $J_{\Omega}$, and hence $(u, v) \in$ $\mathcal{N}_{\Omega}$.

On the other hand, by the fact $(u_{k}, v_{k}) \in \mathcal{N}_{\lambda_{k}}$ and (\ref{f4}), we have
$$
\begin{aligned}
m_{\lambda_k}=&J_{\lambda_{k}}(u_{k}, v_{k}) \\
= & \left(\frac{1}{p}-\frac{1}{2\gamma}\right)\int_{V} \left(R_\alpha\ast F(u_k,v_k)\right) F(u_k, v_k)\mathrm{~d}\mu\\=&\left(\frac{1}{p}-\frac{1}{2\gamma}\right)\int_{V} \left(R_\alpha\ast F(u,v)\right) F(u, v)\mathrm{~d}\mu+o_k(1)\\=&\left(\frac{1}{p}-\frac{1}{2\gamma}\right)\int_{\Omega_{a} \cup \Omega_{b}} \left(R_\alpha\ast F(u,v)\right) F(u, v)\mathrm{~d}\mu+o_k(1)\\=&J_{\Omega}(u,v)+o_k(1).
\end{aligned}
$$
By Lemma \ref{pb}, we get that $J_{\Omega}(u, v)=m_{\Omega}$. Thus $(u, v)$ is a ground state solution to the system (\ref{01}). 

By the facts $(u_{k}, v_{k}) \in \mathcal{N}_{\lambda_{k}}$, $(u, v) \in \mathcal{N}_{\Omega}$, the Br\'{e}zis-Lieb lemma and (\ref{f4}), one concludes that
$$\begin{aligned}
& \|(u_k-u,v_k-v)\|^p_{\lambda_k}\\=&\int_{V}\left(|\nabla (u_k-u)|^{p}+(\lambda_ka+1)|u_k-u|^p\right) \mathrm{~d} \mu+ \int_{V}\left(|\nabla (v_k-v)|^{p}+(\lambda_kb+1)|v_k-v|^p\right) \mathrm{~d} \mu\\=&\int_{V}\left(|\nabla u_k|^{p}+(\lambda_ka+1)|u_k|^p\right) \mathrm{~d} \mu-\int_{V}\left(|\nabla u|^{p}+(\lambda_ka+1)|u|^p\right) \mathrm{~d} \mu+o_k(1)\\&+\int_{V}\left(|\nabla v_k|^{p}+(\lambda_kb+1)|v_k|^p\right) \mathrm{~d} \mu-\int_{V}\left(|\nabla v|^{p}+(\lambda_kb+1)|v|^p\right) \mathrm{~d} \mu+o_k(1)\\=&\int_{V} \left(R_\alpha\ast F(u_k,v_k)\right)F(u_k, v_k)  \mathrm{~d} \mu-\int_{V} \left(R_\alpha\ast F(u,v)\right)F(u, v)  \mathrm{~d} \mu+o_k(1)\\=&o_k(1). 
\end{aligned}$$
Hence we complete the proof.
\qed

\section{Acknowledgements}
The author would like to thank the referees for helpful comments and suggestions on this paper. Moreover, the author is supported by the National Natural Science Foundation of China, no.12401135.

\
\

{\bf Data availability :} Not applicable.

\
\

{\bf Declarations}

\
\

{\bf Conflict of interest:} The author declares that there are no conflicts of interests regarding the publication of
this paper.

\end{document}